\documentclass[11pt,oneside]{amsart}
\usepackage{graphicx} % Required for inserting images

\usepackage{latexsym}
\usepackage{tikz-cd}
\usepackage{geometry}
\usepackage[english]{babel}
\usepackage[T1]{fontenc}
\usepackage[utf8]{inputenc}
\usepackage{graphicx}
\usepackage{graphics}
\usepackage{mathrsfs}
\usepackage{amssymb}
\usepackage{amstext}
\usepackage{times}
\usepackage{cancel}
\usepackage{epsfig}
\usepackage{yfonts}
\usepackage{amsthm}
\usepackage{amsfonts}
\usepackage{amsmath}
\theoremstyle{plain}
\newtheorem{teorema}{Theorem}[section]
\newtheorem{lemma}[teorema]{Lemma}

\newtheorem{corollario}[teorema]{Corollary}
\newtheorem{proposizione}[teorema]{Proposition}
\theoremstyle{definition}
\newtheorem{definizione}[teorema]{Definition}
\theoremstyle{remark}
\newtheorem{remark}[teorema]{Remark}
\newtheorem{esempio}[teorema]{Example}

\def \p {\partial}
\def \bpa {\overline{\p}}

\title[A Kummer construction for Chern-Ricci flat balanced manifolds]{A Kummer construction for Chern-Ricci flat balanced manifolds}

\author{Federico Giusti}
\address{Department of Mathematics, Aarhus University, Ny Munkegade 118, 8000 Aarhus C, Denmark}
\email{federico.giusti@math.au.dk}
\author{Cristiano Spotti}
\address{Department of Mathematics, Aarhus University, Ny Munkegade 118, 8000 Aarhus C, Denmark}
\email{c.spotti@math.au.dk}
\date{\today}


\begin{document}

\begin{abstract} 
Given a non-Kähler Calabi-Yau compact orbifold with isolated singularities endowed with a Chern-Ricci flat balanced metric, we study, via a gluing construction, the existence of Chern-Ricci flat balanced metrics on its crepant resolutions, and discuss applications to the search of solutions for the Hull-Strominger system. We also describe the scenario of singular threefolds with ordinary double points, and see that similarly is possible to obtain balanced approximately Chern-Ricci flat metrics.
\end{abstract}

\maketitle

\section{Introduction}

With the ultimate aim of geometrizing and classifying, one of the most studied problems in complex geometry is the existence of hermitian metrics that can be regarded as \textit{special}. Through the years, the Kähler case is the one that has been studied and understood the most, however, in the last decades the interest towards the non-Kähler world has been increasing more and more, leading to the search for special metrics also in this particular context. While in the Kähler case special metrics arise naturally, the non-Kähler scenario is too wild to guide us directly towards some central notion of special metric. Nevertheless, one can have indications on the path to follow by watching the Kähler world; more specifically, given an $n$-dimensional complex manifold $(M,J)$, if it is Kähler the obvious class of special (on a first level) metrics is given exactly by Kähler metrics - which we recall being hermitian metrics $h$ whose fundamental form $\omega:=h(J\cdot,\cdot)$ is $d$-closed. In addition, this condition can also be combined with the notion of \textit{Einstein metric} (thanks to the properties of Kähler metrics) from the general riemannian case, giving rise to the notion of Kähler-Einstein metrics, which are universally regarded as the "most special" in the Kähler world. Likewise, other notions of special Kähler metrics have been introduced and studied (some of them are still central in the study of Kähler geometry), like \textit{constant scalar curvature Kähler} (cscK) metrics, or the more general class of \textit{extremal Kähler} metrics (introduced by Calabi in \cite{C}), however they all share the fact that they are giving a curvature condition on the metric, thus this suggests that when searching for special metrics in the non-Kähler case we shall ask for these metrics to be special under two aspects: the cohomological one (satisfying a condition possibly generalizing the Kähler one) and the curvature one.\\ 
Regarding the cohomological aspect, several conditions have been introduced that generalize the Kähler one, and one of the most studied is given by $d \omega^{n-1}=0$, identifying the class of \textit{balanced metrics} (originally introduced by Michelsohn \cite{M}, and also considered by Gauduchon in \cite{G} as \textit{semi-Kähler} metrics), which is the class of metrics we are interested in working with. Balanced metrics carry many interesting properties such as the coincidence between the Hodge laplacian and the Dolbeault laplacian on scalar functions (showed by Gauduchon in \cite{G}), or the preservation of the balanced condition for manifolds under holomorphic submersions proved in \cite{M} (showing a sort of duality between the Kähler condition and the balanced condition). Also in \cite{M}, Michelsohn proved a characterization of balanced metrics in terms of currents, which leads to the celebrated result from Alessandrini and Bassanelli (see \cite{AB1}) showing that the class of compact balanced manifolds is closed under proper modifications (condition not satisfied by the class of Kähler manifolds). Moreover, balanced metrics ended up being central in many interesting currently open problems, such as the conjecture from Fino and Vezzoni (see \cite{FV}) and the Gauduchon conjecture for balanced metrics (see \cite{STW}, in which was solved in its original version for Gauduchon metrics - identified by the condition $\partial\overline{\partial}\omega^{n-1}=0$, which weakens the balanced condition - posed by Gauduchon).\\
Moving instead on the curvature aspect, there are several known notions of special metrics in the non-Kähler world such as Chern-Ricci flat metrics, Bismut-Ricci flat metrics (which in the balanced case are equivalent to Chern-Ricci flat metrics, see \cite{AI}), Chern-Einstein metrics and many more.\\
The main goal of this paper is to study the existence of Chern-Ricci flat balanced metrics on the crepant resolutions of certain non-Kähler Calabi-Yau singular manifolds endowed with a singular Chern-Ricci flat balanced metric. More specifically we wish to work on compact orbifolds $\tilde{M}$ whose singular set is made of isolated singularities admitting crepant resolutions, and are endowed with a balanced Chern-Ricci flat singular metric $\tilde{\omega}$. Then, performing a cut-off on the singular metric to the flat one around the singularities, together with what it is known on orbifold singularities (i.e. Joyce's theory on ALE spaces) and its crepant resolutions to build, with a gluing construction (inspired by, for example, \cite{AP}, \cite{BM} and \cite{J}), Chern-Ricci flat balanced metrics on the crepant resolutions of the orbifold, provided an assumption on the linearization of the operator. The strategy of the proof consists of two main steps: (1) a metric "rough" gluing between the singular Chern-Ricci flat balanced metric $\tilde{\omega}$ with the (rescaled) Joyce's ALE metrics $\omega_{ALE}$ (that are Kähler Calabi-Yau metrics on the crepant resolution of the singularity model, see \cite{J}), and (2) an Implicit Function Theorem deformation argument, where the deformation is a balanced deformation (introduced in \cite{FWW}) and all the analysis is performed in suitable weighted Hölder spaces (by the way, the choice of the deformation shows also that on this class of manifolds a Calabi-Yau-type theorem for balanced metrics holds for some classes in the balanced cone). Our main result is the following.
\begin{teorema}\label{main}
Let $(\tilde{M},\tilde{\omega})$ be an $n$-dimensional non-Kähler Calabi-Yau compact orbifold with isolated singularities, endowed with $\tilde{\omega}$ a singular Chern-Ricci flat balanced metric, and let $M$ be a crepant resolution of $\tilde{M}$. Suppose that the operator $\Delta_\omega+|\p\omega|_\omega^2Id$ has vanishing kernel on $\tilde{M}$. Then $M$ admits a Chern-Ricci flat balanced metric such that
\[
[\omega^{n-1}]=[\hat{\omega}^{n-1}]=[\tilde{\omega}^{n-1}]+(-1)^{n-1}\varepsilon^{(2n-2)}(\sum_{i=1}^{k_j} \sum_{j=1}^k a_j^i PD[E_j^i])^{n-1},
\]
where $PD[E_j^i]$ denotes the Poincaré dual of the class $[E_j^i]$, represented by the $i$-th component of the $j$-th exceptional divisor arising from the resolution of the corresponding singularity.
\end{teorema}
Our interest towards Chern-Ricci flat balanced metrics comes actually from the realm of Calabi-Yau geometry. Indeed, for a not necessarily Kähler Calabi-Yau manifold (i.e. a complex manifold endowed with a holomorphic volume form) it was introduced by Hull and Strominger (respectively in \cite{Hu} and \cite{S}) a system of four equations coming from superstring theory known as the \textit{Hull-Strominger system}, whose solutions have proved to be extremely hard to construct (see \cite{GF} for a full presentation of the system and some known solutions, together with several other references such as \cite{AGF}, \cite{FuY}, \cite{LY3}, \cite{P}, \cite{TY} and the very recent \cite{CPY2}, \cite{FeY} for the invariant case, \cite{PPZ} for a flow approach, and the recent moment map picture from \cite{GFGM}). The problem of solving this system, apart from its physical meaning, carries great geometric interest, since it generalizes the Calabi-Yau condition to the non-Kähler framework, and it holds a central role in the geometrization conjecture for compact Calabi-Yau threefolds known as \textit{Reid's Fantasy} (see \cite{R}). This last conjecture, in particular, states that all compact Kähler Calabi-Yau threefolds can be connected through a finite number of \textit{conifold transitions} (introduced by Clemens and Friedman, see \cite{F}), i.e. a procedure consisting of the contraction of a finite family of disjoint $(-1,-1)$-curves in a compact Calabi-Yau threefold, followed by the smoothing of the ordinary double points obtained from the previous step. These framework motivates further our interest towards Chern-Ricci flat balanced metrics, since it is directly related to one of the equation of the Hull-Strominger system, namely the \textit{conformally balanced equation}, which on a compact Calabi-Yau manifold $(X,\Omega)$ - where $\Omega$ is the holomorphic volume form - is an equation for hermitian metrics (actually their fundamental forms) $\omega$ given by $d(||\Omega||_\omega \omega^{n-1})=0$ which is clearly satisfied by balanced Chern-Ricci flat hermitian metrics. Moreover, our result takes a first step towards solving the problem proposed by Becker, Tseng and Yau in (Section 6 of) \cite{BTY}. \\
A natural question that arises from this construction, in the setting of the Hull-Strominger system and Reid's Fantasy, is if this strategy can be adapted to the case of singular threefolds with a finite family of ordinary double points aiming (in some sense) towards "reversing the arrow" in the construction done by Fu, Li and Yau in \cite{FLY} and Collins, Picard and Yau in \cite{CPY1}. Our strategy in this scenario unfortunately carries a complication that is hidden in the asymptotic behaviour of the standard Calabi-Yau metric $\omega_{co,a}$ (introduced by Candelas and de la Ossa, see \cite{CO}) on the small resolution of the standard conifold, thus in the last section we shall talk in more detail about the difficulties of the case and discuss some possible paths towards a solution of the problem in this other scenario. We are anyway able to achieve some partial result, that is the following.
\begin{proposizione}\label{balancedodp}
Let $(\tilde{M},\tilde{\omega})$ be a smoothable projective Kähler Calabi-Yau nodal threefold (with $\tilde{\omega}$ a singular Calabi-Yau metric), and let $M$ be a compact (not necessarily Kähler) small resolution of $\tilde{M}$. Then $M$ admits a balanced approximately Chern-Ricci flat metric $\omega$ in the class $[p^*\tilde{\omega}^2]$, where $p$ is the resolution map collapsing the $\mathbb{P}^1$'s.
\end{proposizione}
The paper is structured as follows. In Section 2, after giving examples, we present the first step of our work, consisting of the construction of a balanced metric on the crepant resolution made with the objects previously introduced, together with the construction of a global holomorphic volume, in order to express the Chern-Ricci potential for our new balanced metric and obtain estimates for it. In Section 3, we apply a deformation argument to obtain, under an assumption on the linearization of the operator, a genuine Chern-Ricci flat balanced metric and we discuss its possible applications to the search of solutions for the Hull-Strominger system. In the last section, i.e. Section 4, we take a look at the case of Ordinary Double Points on threefolds, walk again through the gluing process from Section 2 to produce again a balanced approximately Chern-Ricci flat balanced metrics, and discuss the difficulties that arise if we try to repeat the deformation argument in this case.

\par 
\vspace{0.5cm}
{\bf Acknowledgements.}
Both the authors are supported by Villum Young Investigator 0019098.
\par 

The authors would like to thank Mario Garcia-Fernandez for useful conversations and remarks. The are also grateful to Song Sun and Junsheng Zhang for pointing out a mistake in the first version of Lemma 2.7's proof. Also they would like to thank the anonymous referee for pointing out Remark 4.8 along with many useful comments that significantly improved the presentation.

\section{The pre-gluing metric}
Following several known gluing constructions from the literature (such as \cite{AP}, \cite{BM}, \cite{J} and many others), our gluing process will be made of two main parts: the construction of a pre-gluing metric (which will be done in this section) obtained from a rough cut-off procedure providing an approximate solution to the problem, and a perturbative argument to obtain a genuine solution.\\
The goal of this section will be to prove the following:
\begin{proposizione}\label{approxbcrf}
Let $\tilde{M}$ be a Calabi-Yau compact orbifold with isolated singularities, endowed with a Chern-Ricci flat balanced singular metric $\tilde{\omega}$, and suppose that it admits $M$ a crepant resolution. Then $M$ admits $\omega$ an approximately Chern-Ricci flat balanced metric.
\end{proposizione}

\subsection{Chern-Ricci flat balanced orbifolds and their crepant resolutions}
Before discussing the construction, we shall establish some notations for the reminder of the paper, and also use the occasion to briefly recall some known results from literature to understand better the framework we will be working in.\\
Throughout the paper we will denote with $\tilde{M}$ an $n$-dimensional non-Kähler Calabi-Yau compact orbifold, i.e. a complex orbifold endowed with a holomorphic volume form $\tilde{\Omega}$, with isolated singularities, such that it admits a crepant resolution $M$. 
\begin{remark}
A necessary condition for an orbifold to admit crepant resolutions is that the isotropy groups corresponding to the singularities are subgroups of $SL(n,\mathbb{C})$, and for $n=3$ it is also sufficient (see \cite{J}), making it a useful criterion to search for examples.
\end{remark}
\begin{remark}\label{exceptional}
The exceptional set of a crepant resolution of an orbifold singularity is always divisorial, i.e. in codimension 1. Indeed, it is known that orbifold singularities are "mild", meaning that (see for example \cite{KM}) every orbifold is normal and $\mathbb{Q}$-factorial. If we now assume by contradiction the existence of a (quasi-projective) small resolution $f:M \rightarrow \tilde{M}$ of an orbifold, and take a Cartier divisor $H \subseteq M$ that is non-trivial on a curve contracted by $f$, we can take $m \in \mathbb{Z}\setminus\{0\}$ such that $mf_*H$ is a Cartier divisor (since $\tilde{M}$ is $\mathbb{Q}$-factorial). But now the smallness of the resolution implies that $mH$ and $f^*(mf_*H)$ coincide, giving in particular that $mH$ is a pullback through $f$, which is a contradiction since $H$ was chosen to be non-trivial on a curve contracted by $f$.
\end{remark}

We will also assume that $\tilde{M}$ is equipped with a singular balanced Chern-Ricci flat metric $\tilde{\omega}$, and thus it is worth giving examples of spaces that satisfy our assumptions, in order to ensure that we are working on an actually existing class of spaces.

\begin{esempio}
A first, trivial example is the one of quotients of tori with isolated orbifold singularities of the form $\mathbb{C}^3/\mathbb{Z}_3$. In these cases, we know that the quotient is equipped with a singular Kähler Calabi-Yau metric, and D. Joyce (in \cite{J}, for example) has shown that also their crepant resolutions admit Kähler Calabi-Yau metrics, which can be obtained via gluing construction in the same fashion as the one we are about to present. However, since every Kähler Ricci-flat metric is also balanced Chern-Ricci flat, we can still consider these spaces in our class, and - as we will se ahead - our construction does not ensure that the Chern-Ricci flat balanced metric obtained need to coincide with the Kähler Calabi-Yau, since the cohomology class preserved is going to be the balanced one, on which there are no known uniqueness results.\\
A possible variation on this argument could be to apply the (orbifold version of) the result of Tosatti and Weinkove in \cite{TW1}, which ensures us that we can find a Chern-Ricci flat balanced metric on the singular quotient of the torus, and thus provides a suitable metric for our construction.
\end{esempio}

\begin{esempio}\label{bty}
A more interesting example can be obtained on torus bundles on some algebraic K3 surfaces. Indeed, Goldstein and Prokushkin produced in \cite{GP} a family of $T^2$ bundles on $K3$ surfaces that do not admit Kähler metrics; and they showed that these threefolds can be endowed with a balanced Chern-Ricci flat metric of the form
\[
\eta=\pi^*\eta_{K3}+\frac{i}{2}\theta\wedge\overline{\theta},
\]
where $\eta_{K3}$ is the Calabi-Yau metric on the $K3$, and $\theta$ is a $(1,0)$-form arising from the duals of the horizontal lift of the coordinate vector fields on the $K3$. These bundles $X$ inherit also a non-Kähler Calabi-Yau structure, i.e. a holomorphic volume form given by 
\[
\Omega=\Omega_{K3}\wedge \theta.
\]
Now, while these are the building blocks of the Fu and Yau solutions for the Hull-Strominger system (see \cite{FuY}), Becker, Tseng and Yau constructed (in \cite{BTY}, Section 6) a $\mathbb{Z}_3$ action on a subclass of the aformentioned torus bundles for some special choices of algebraic $K3$'s, of the form
\[
\rho:(z_0,z_1,z_2,z_3,z_4,z )\longrightarrow (\zeta^2 z_0,\zeta^2 z_1,\zeta z_2,z_3,z_4,\zeta^2 z),
\]
with $\zeta$ a cubic root of unity different from $1$, and where the $z_i$s are the homogeneous coordinates of the $\mathbb{P}^3$ in which the $K3$ lies, and $z$ is the fiber coordinate. This action, despite not preserving the Calabi-Yau structures of the base and the fibres, it preserves $\Omega$, together with the Chern-Ricci flat balanced metric $\eta$, producing an orbifold with $9$ isolated singularities of the form $\mathbb{C}^3/\mathbb{Z}_3$, i.e. exactly from the family of orbifolds we are interested in working with.
\end{esempio}

\begin{esempio}\label{st}
A further example comes from an action of $\mathbb{Z}_4$ on the Iwasawa manifold, constructed by Sferruzza and Tomassini in \cite{ST}. In said paper they showed that the action of $Z_4=\langle \sigma \rangle$ on $\mathbb{C}^3$, where
\[
\sigma(z_1,z_2,z_3):=(iz_1,iz_2,-z_3),
\]
descends to the quotient corresponding the (standard) Iwasawa manifold, producing 16 isolated singular points. Moreover, if we recall the standard coframe of invariant (with respect to the Heisenberg group operation) 1-forms
\[
\varphi_1:=dz_1, \quad \varphi_2:=dz_2, \quad \varphi_3:=dz_3-z_2dz_1,
\]
this can be used to construct a balanced metric $$\omega:=\frac{i}{2}(\varphi_1\wedge\varphi_{\bar 1}+\varphi_2\wedge\varphi_{\bar 2}+\varphi_3\wedge\varphi_{\bar 3}),$$ which descends to a Chern-Ricci flat balanced metric on the Iwasawa manifold, and is clearly invariant through $\sigma$, as well as the standard holomorphic volume of $\mathbb{C}^3$. Thus the quotient of the Iwasawa manifold through this action gives again an orbifold satisfying our hypotheses. 
\end{esempio}

Our aim is to work on the crepant resolution $M$, and obtain via a gluing construction (using Joyce's ALE metrics on the bubble, see \cite{J}) a family of Chern-Ricci flat balanced metrics from $(\tilde{M},\tilde{\omega})$.
In the following we will focus on the construction of the pre-gluing metric on $M$, that will be an \textit{approximately} Chern-Ricci flat balanced metric. To make the presentation more clear, we will divide the process into three natural steps, and for simplicity assume that $\tilde{M}$ has just one singularity (the process obviously applies analogously to the case in which the singularities are any finite number). We are also going to compute explicitly a holomorphic volume form for $M$ (starting from the one on $\tilde{M}$), since such form is a crucial ingredient for the deformation argument in the following section, as it can be used to obtain a global expression for the Chern-Ricci potential.

\subsection{Pre-gluing - Step 1}
We first glue together the metric $\tilde{\omega}$ with the flat metric $\omega_o$ centered at the singularity so that the resulting metric is balanced. This follows actually from the following remark, which holds for any balanced manifold and recovers a weaker version of the strategy used with normal coordinates in the Kähler case.

\begin{lemma}\label{flatcutoff}
Given $(X,\eta)$ an $n$-dimensional balanced orbifold with isolated singularities, for every $x \in X$ it exists a sufficiently small $\varepsilon>0$, coordinates $z$ centered at $x$ and a balanced metric $\eta_\varepsilon$ such that
\[
\eta_\varepsilon=\begin{cases} \omega_o & \text{if} \>\> |z|<\varepsilon \\ \eta & \text{if} \>\> |z|> 2\varepsilon\end{cases},
\]
where $\omega_o$ is the flat metric around $x$, and such that $|\eta_\varepsilon-\omega_o|_{\omega_o}<c\varepsilon$ on $\{\varepsilon \leq |z| \leq 2\varepsilon\}$.
\end{lemma}
\begin{proof}
If $(X,\eta)$ is an $n$-dimensional balanced orbifold and we fix any point $x \in M$, we can choose coordinates $z$ around $x$ such that, in a sufficiently small neighborhood of the point, we can write
\[
\eta=\omega_o+O(|z|),
\]
where $\omega_o$ is the flat metric in a neighborhood of $x$ in the coordinates $z$, and $O(|z|)$ is some form decaying at least linearly at the point $x$. But now this means that if we take the $n-1$ power we obtain
\[
\eta^{n-1}=\omega_o^{n-1}+\alpha,
\]
where $\alpha$ is a closed $(n-1,n-1)$-form (thanks to the facts that $\eta$ is balanced and $\omega_o$ is Kähler) such that $\alpha=O(|z|)$. Thus if we restrict to a simply connected neighborhood of $x$, it exists a real $(n-2,n-2)$ form $\beta$ such that 
\[
\alpha=i\p\bpa\beta,
\]
and it can be chosen to be such that $\beta=O(|z|^3)$. Hence, if we introduce a cut-off function
\[
\chi(y):=
\begin{cases} 
0 & \text{if} \>y\leq 1 \\ 
\text{non decreasing} & \text{if} \> 1<y<2 \\ 
1 & \text{if} \> y \geq 2
\end{cases}
\]
and call $r(z):=|z|$ the (flat) distance from $x$, we can take $\chi_\varepsilon(y):=\chi(y/\varepsilon)$ and define
\[
\eta_\varepsilon^{n-1}:=\omega_o^{n-1}+i\p\bpa(\chi_\varepsilon(r)\beta).
\]
Here, the notation $\eta_\varepsilon^{n-1}$ makes sense thanks to \cite{M}, since on the gluing region holds
\[
|i\p\bpa(\chi_\varepsilon(r)\beta)|\leq |i\p\bpa\chi_\varepsilon||\beta|+|\p\chi_\varepsilon||\p\beta|+|\chi_\varepsilon||i\p\bpa\beta|\leq c\varepsilon,
\]
ensuring that $\eta_\varepsilon^{n-1}>0$. Thus we have obtained a balanced metric $\eta_\varepsilon$ on $X\setminus\{x\}$ which is exactly flat in a neighborhood of $x$. The same argument applies to the orbifold points after taking a cover chart.
\end{proof}
This also shows that, for example, there is a \textit{canonical} choice of balanced metric on the blow-up at a point of a balanced manifold, since thanks to this construction, any balanced metric can be glued to the Burns-Simanca metric preserving the balanced condition.\\ 

Thus we can start from our Chern-Ricci flat balanced metric $\tilde{\omega}$ on $\tilde{M}$ and obtain the corresponding cut-off metric $\tilde{\omega}_\varepsilon$ in a neighborhood of the orbifold singularity $x$ by chosing coordinates $z$ on the orbifold cover chart. For our construction, it will however be more convenient to slightly vary the cut-off function and, for $p>0$, choose
\[
\chi_{\varepsilon,p}(y):=\chi(y/\varepsilon^p)
\]
so that the gluing region for $\tilde{\omega}_\varepsilon$ becomes $\{\varepsilon^p < r <2\varepsilon^p\}$.
Also, using again the results in \cite{M}, we can notice that, even though we are cutting at the level of $(n-1,n-1)$-forms, we have that on the gluing region the metric keeps being close to the flat metric, indeed:
\begin{remark}\label{michelsohn}
Notice that we can choose a basis $\{e_j\}$ of $1$-forms diagonalizing simultaneously $\omega_o$ (we can actually assume it to be the identity) and $\tilde{\omega}_\varepsilon$; this means that also $\omega_o^{n-1}$ and $\tilde{\omega}_\varepsilon^{n-1}$ are diagonal (in the sense of $(n-1,n-1)$-forms, implying that also the term $O(r)$ is necessarily diagonal with respect to this basis. Thus we can write
\[
\tilde{\omega}_\varepsilon^{n-1}=\sum_{j=1}^n (1+O(r))\widehat{e_j\wedge Je_j}
\]
and applying Michelson's result with $\Lambda_j=1+O(r)$, we obtain $\tilde{\omega}_\varepsilon=\sum_{j=1}^n\lambda_j e_j\wedge Je_j$, with
\[
\lambda_j=\frac{((1+O(r))\cdots (1+O(r)))^{\frac{1}{n-1}}}{1+O(r)}=1+O(r),
\]
which implies, again thanks to Michelson's theorem
\[
\omega=\sum_{j=1}^n\left(1+O(r)\right) e_j\wedge Je_j =\omega_o+O(r),
\]
showing also that $d\omega$ has uniformly bounded norm.
\end{remark}

\subsection{Pre-gluing - Step 2}
In this second step we instead perform the gluing between Joyce's Kähler-Ricci flat ALE metric $\omega_{ALE}$ recalled in Section 2 and the flat metric $\omega_o$ of $\mathbb{C}^n$, on the crepant resolution $\hat{X}$ of the singular model $\mathbb{C}^n/G$, and we will actually be able to do it without losing the Kähler condition. To do this we recall that away from the singularity holds
\[
\omega_{ALE}=\omega_o+Ai\partial\overline{\partial}(r^{2-2n}+o(r^{2-2n})),
\]
where $A>0$ is a constant and $r$ is the (flat) distance from the singularity.
This suggests introducing a large parameter $R$ and a smooth cut-off function $\chi_R(x):=\chi_2(x/R)$ on $[0,+\infty)$ such that
\[
\chi_2(y):=
\begin{cases} 
1 \qquad & \text{if} \> y\leq \frac{1}{4}, \\ 
\text{Non increasing} \quad & \text{if} \> \frac{1}{4} < y < \frac{1}{2} , \\ 
0 \qquad & \text{if} \> y\geq \frac{1}{4}, 
\end{cases}
\]
from which we introduce the family of closed $(1,1)$-forms
\[
\omega_R=\omega_o+i\partial\overline{\partial}(\chi_R(r)(r^{2-2n}+o(r^{2-2n})).
\]
Once again,  on the gluing region $G_R:=\{\frac{R}{4}\leq r \leq \frac{R}{2}\}$ we have
\[
|\omega_R- \omega_o|_{\omega_o} \leq |i\partial\overline{\partial}(\chi_R(r)(r^{2-2n}+o(r^{2-2n}))|_{\omega_o} \leq cR^{-2n} \leq cr^{-2n},
\]
which clearly implies the positivity of $\omega_R$ also on $G_R$ (as long as $R$ is chosen to be sufficiently large) ensuring that $\omega_R$ is a Kähler metric on $\hat{X}$ which is exactly flat outside of a compact set.

\subsection{Pre-gluing - Step 3}
In this third and last step we want to glue together the metrics $\tilde{\omega}_\varepsilon$ from Step 1 with the metric $\omega_R$ from Step 2 by matching isometrically the exactly conical regions. In order to do this we are going to need to rescale by a constant $\lambda>0$ the metric on $\hat{X}$, and we will now see that this constant is a geometric constant, since it is dictated by the geometries of the two metrics we are gluing together.\\
In what follows we will denote with $z$ the coordinates on $M_{reg}$ nearby the singularity and with $\zeta$ the coordinates on $\hat{X}$, both given by the identification with the singularity model $\mathbb{C}^n/G$. We then consider the regions
\[
C_R:=\{R/4\leq r(\zeta) \leq 2R\} \subseteq \hat{X} \qquad \text{and} \qquad C_\varepsilon := \{\varepsilon^p/4 \leq r(z) \leq 2\varepsilon^p \} \subseteq M_{reg}
\]
and define a biholomorphism between them by imposing
\[
\zeta=\left(\frac{R}{\varepsilon^p}\right) z.
\]
From this expression we have that on the identified region the following identity holds
\[
r(\zeta)=r\left(\left(\frac{R}{\varepsilon^p}\right)z\right)=\frac{R}{\varepsilon^p}r(z)
\]
which yields $\lambda=\lambda(\varepsilon,R):=\left(\frac{\varepsilon^p}{R}\right)^2$. From this follows $\lambda r^2(\zeta)=r^2(z)$, and thus on the identified conical regions $C_R':=\{R\leq r(\zeta) \leq 2R\}\simeq  \{\varepsilon^p \leq r(z) \leq 2\varepsilon^p \}=:C_\varepsilon'$ holds
\[
\lambda \omega_o(\zeta)=\omega_o(z), \qquad \text{and consequently} \qquad \lambda \omega_R = \tilde{\omega}_\varepsilon.
\]
Hence, $\lambda$ is the needed rescaling factor, which allows us to define the glued family of balanced metrics on the crepant resolution $M$ as
\[
\omega_{\varepsilon,R}:=
\begin{cases} \lambda\omega_R \qquad & \text{on} \> r(\zeta)\leq R, \\
\omega_o \qquad & \text{on}\> \varepsilon^p \leq r(z) \leq 2\varepsilon^p, \\ 
\tilde{\omega}_\varepsilon \qquad & \text{on} \> r(z) \geq 2\varepsilon^p.
\end{cases}
\]

\begin{remark}
Notice that this first construction implies an Alessandrini-Bassanelli type result (see \cite{AB1}) since it shows that any compact complex manifold bimeromorphic to a balanced orbifold with isolated singularities is also balanced.
\end{remark}

In order to understand better the geometry of this new family of metrics, we shall obtain again some estimates on its distance from the flat metric on the gluing region, and since inside said region there is also an exactly flat part - whose geometry is also understood - which separates the two gluing regions from the first two steps, we can just estimate the distance separately on the two regions from the previous steps and then take the maximum.\\
Clearly, the metric is unaltered on the gluing region from Step 1, thus we still have on $G_\varepsilon$ that
\[
|\nabla_{\omega_o}^k (\omega -\omega_o)|_{\omega_o}\leq cr^{1-k},
\]
for all $k \geq 0$.\\
On the other hand, since in this step we had to rescale the metric on $\hat{X}$, we have to check how it has affected the distance from the cone. To have clearer estimates, we will express also this one in terms of the small coordinates $z$, and we will relate the parameters $R$ and $\varepsilon$ by chosing $R=\varepsilon^{-q}$, with $q>0$. We first notice that on $G_R$ (actually the corresponding region through the biholomorphism) it holds
\[
\begin{aligned}
\langle \omega_{\varepsilon,R}-\omega_o,\omega_{\varepsilon,R}-\omega_o\rangle_{\omega_o}(z)= & \> \lambda^{-2}\langle \lambda(\omega_R-\omega_o),\lambda(\omega_R-\omega_o)\rangle_{\omega_o}(\zeta) \\
= & \>\langle \omega_R-\omega_o,\omega_R-\omega_o\rangle_{\omega_o}(\zeta)
\end{aligned}
\] 
implying that $|\omega_{\varepsilon,R}-\omega_o|_{\omega_o}(z)=|\omega_R-\omega_o|_{\omega_o}(\zeta)$. From here, we can recall the estimate done in Step 2 and obtain
\[
\begin{aligned}
|\omega_{\varepsilon,R}-\omega_o|_{\omega_o}(z)\leq & |\omega_R-\omega_o|_{\omega_o}(\zeta) \\ 
 \leq & \> cr^{-2n}(\zeta) = c\varepsilon^{2nq} \leq c r^{2nq/p}(z).
\end{aligned}
\]
which implies, on the whole gluing region, that for all $k \geq 0$ holds
\[
|\nabla_{\omega_o}^k (\omega_{\varepsilon,R} - \omega_o)|_{\omega_o} \leq cr^{m-k},
\]
where $m=\min\{1, 2nq/p\}$.\par 
From now on, to simplify the notation, we shall denote $\omega_{\varepsilon,R}$ just with $\omega$.

\subsection{The Chern-Ricci potential}

In order to use this description of the metrics to estimate the Chern-Ricci potential on the gluing region we are also going to need to understand how the holomorphic volume form of the resolution is related to the holomorphic volume of our background Calabi-Yau orbifold.\\
Before doing it we start by fixing some notation. Denote
\begin{itemize}
\item[•] with $\tilde{\Omega}$ the holomorphic volume of $M_{reg}$ such that $$\tilde{\omega}^n=i\tilde{\Omega}\wedge\overline{\tilde{\Omega}};$$
\item[•] with $\hat{\Omega}$ the rescaled holomorphic volume of the singularity model $\mathbb{C}^n/G$ (and its crepant resolution $\hat{X}$) in order to match the metric rescaling, i.e. $\hat{\Omega}:=\lambda^{n/2}\Omega_o$ where $$(\omega_{ALE})^n=i\Omega_o\wedge\overline{\Omega}_o.$$
\end{itemize}

Now, in a neighborhood of the singularity it exists a holomorphic function $h$ such that $$\tilde{\Omega}=h\Omega_o.$$ 
On the other hand, under the rescaling biholomorphism that glues $\hat{X}$ to $M\setminus\{x\}$, we identify the $\Omega_o$ around the singularity with $\hat{\Omega}$, thus we can read $h$ as a holomorphic function on the singularity model, and hence holomorphically extend it to a holomorphic function on the whole $\hat{X}$, and thus we can glue together $h\hat{\Omega}$ with $\tilde{\Omega}$ to obtain $\Omega$ a holomorphic volume for $M$.\\
We can also obtain information on $h$ by noticing that, since $\tilde{\omega}$ is asymptotic to $\omega_o$ around the singularity, we obtain that around $x$ it holds
\[
(1+O(|z|))\omega_o^n=\tilde{\omega}^n=i\tilde{\Omega}\wedge \overline{\tilde{\Omega}}=|h|^2i\Omega_o\wedge
\overline{\Omega}_o=|h|^2\omega_o^n
\]
from which follows
\[
|h|=1+O(r),
\]
from which, by continuity, we have that $|h|^2\equiv 1$ on the exceptional part.\\ 
Thus we can define a global Chern-Ricci potential as
\[
f=f_{p,q,\varepsilon}:=\log\left(\frac{i\Omega\wedge\overline{\Omega}}{\omega^n}\right).
\]
and conclude this section by describing the behaviour of $f$ in all the regions of $M$, to show that it is suitable to apply a deformation argument similar to the one done in \cite{BM}. We have

\begin{itemize}
\item[•] on $\{r(z) > 2\varepsilon^p\}$ hold $\omega=\tilde{\omega}$ and $\Omega=\tilde{\Omega}$, thus $f\equiv 0$;
\item[•] on $\{\varepsilon^p \leq r(z) \leq 2\varepsilon^p\}$ hold $\omega=\omega_o+O(r)$ and $\Omega\wedge\overline{\Omega}=\Omega_o\wedge \overline{\Omega_o}+O(r)$, from which we have
\[
f=\log\left(\frac{\omega_o^n+O(r)}{\Omega_o\wedge \overline{\Omega_o}+O(r)}\right)=\log(1+O(r))=O(r);
\]
\item[•] on $\{\frac{1}{2}\varepsilon^p \leq r(z) \leq \varepsilon^p\}$ hold $\omega=\omega_o$ and $\Omega\wedge\overline{\Omega}=i(1+O(r))\Omega_o\wedge \overline{\Omega_o}$, from which follows $f=O(r)$;
\item[•] on $\{\frac{1}{4}\varepsilon^p/2 \leq r(z) \leq \frac{1}{2}\varepsilon^p\}$ hold $\omega=\omega_o+O(r^{2nq/p})$ and $\Omega\wedge\overline{\Omega}=\Omega_o\wedge \overline{\Omega_o}+O(r)$, implying $f=O(r^m)$;
\item[•] on $\{r(z) < \varepsilon^p/2 \}$ hold $\omega^n=i\Omega_o\wedge \overline{\Omega}_o$ and $\Omega\wedge\overline{\Omega}=i(1+O(r))\Omega_o\wedge \overline{\Omega_o}$, giving once again $f=O(r)$.
\end{itemize}
Thus we can write globally (on $M$) that
\[
|f|\leq cr^m,
\]
ensuring that the metric $\omega$ is an \textit{approximately Chern-Ricci flat} balanced metric (as wanted in Proposition \ref{approxbcrf}), hence a suitable one to perform our gluing construction.

\section{The deformation argument}

In this section we will see that what was built in the previous section are exactly the ingredients we need to introduce a deformation argument in the same fashion as \cite{BM}, in order to obtain a balanced Chern-Ricci flat metric on our crepant resolution $M$. We will also analyze the cohomology class of the metric obtained and see why said metric is interested in the framework of the Hull-Strominger system.

\subsection{The strategy}

We will now set up the problem for this section. First of all we recall the deformation of the metric that preserves the balanced condition introduced in \cite{FWW} (here taken with a particular ansatz):
\[
\omega_\psi^{n-1}:=\omega^{n-1}+i\partial\overline{\partial}(\psi\omega^{n-2}), \quad \psi \in C^{\infty}(M,\mathbb{R}) \>\> \text{such that} \>\> \omega_\psi^{n-1} >0.
\]
Thus the problem we are interested in solving, following what was done in \cite{BM}, is the \textit{balanced Monge-Ampère type} equation
\begin{equation}\label{balma}
\omega_\psi^n=e^f\omega^n
\end{equation}
for $\psi \in C^{\infty}(M,\mathbb{R})$ such that $\omega_\psi^{n-1} >0$.
\begin{remark}
The equation introduced above makes sense, because, as we've seen, $f=O(r^m)$, thus $e^f=1+O(r^m)$, meaning that $e^f\omega^n$ is nearby $\omega^n$ itself, hence it makes sense to try to obtain it as a small deformation of $\omega$.
\end{remark}

For practicality, it is useful to reformulate our equation as an operator on the space of smooth functions, thus we introduce $F:C^{\infty}(M,\mathbb{R}) \rightarrow C^{\infty}(M,\mathbb{R})$ as
\[
F(\psi)=F_\varepsilon(\psi):=\frac{\omega_\psi^n}{\omega^n}-e^f.
\]
Our aim is then to solve the equation $F(\psi)=0$ - which is equivalent to \eqref{balma} -  through a fixed point argument, hence the first step to take towards this argument is to compute the linearization at $0$ of the operator $F$. To do this we shall introduce the notation $\omega_0':=\frac{d}{dt}_{|_{t=0}}\omega_{tu}$, where $\omega_{tu}$ is the curve corresponding to the tangent vector $u \in C^{\infty}(M,\mathbb{R})$, and compute the derivative at zero of $\omega_{tu}^n$ in two different ways:
\begin{itemize}
\item[] $$\frac{d}{dt}_{|_{t=0}}\omega_{tu}^n=n\omega^{n-1}\wedge\omega_0';$$
\item[] $$\frac{d}{dt}_{|_{t=0}}\omega_{tu}^n=i\partial\overline{\partial}(u\omega^{n-2})\wedge \omega + \omega^{n-1} \wedge\omega_0'.$$
\end{itemize}
Even though none of these two expressions are explicit, we can put them together to obtain an explicit one for the linearization, that is
\[
\begin{aligned}
Lu:=L_\varepsilon u=d_0F(u)= \frac{n}{n-1} \frac{i\partial\overline{\partial}(u \omega^{n-2})\wedge\omega}{\omega^n}.
\end{aligned}
\]
Here we can work through a few computations to get a clearer (and much more understandable) expression for the operator.
\begin{lemma}\label{balancedoperator}
The linearized operator $L$ can be written as
\begin{equation}\label{linbma}
Lu=\frac{1}{n-1}\left( \Delta_\omega u+\frac{1}{n-1}|\p\omega|_\omega^2u\right) 
\end{equation}
for all $u \in C^\infty(M)$, where $\Delta_\omega$ denotes the Hodge laplacian taken with negative eigenvalues.
\end{lemma}
\begin{proof}
For all $n \geq 3$ it holds
\[
\begin{aligned}
i\partial\overline{\partial}(u\omega^{n-2})= & i\partial(\overline{\partial}u\wedge \omega^{n-2}+(n-2)u\overline{\partial}\omega\wedge \omega^{n-3}) \\
= & i\partial\overline{\partial}u\wedge\omega^{n-2}-(n-2)i\overline{\partial}u\wedge \partial \omega\wedge\omega^{n-3}+(n-2)i\partial u \wedge \overline{\partial}\omega \wedge \omega^{n-3} \\
& + (n-2)ui\partial\overline{\partial}\omega\wedge \omega^{n-3}-(n-2)(n-3)ui\overline{\partial}\omega\wedge\partial\omega\wedge \omega^{n-4},
\end{aligned}
\]
and since the balanced condition $d\omega^{n-1}=0$ implies $\partial\omega\wedge\omega^{n-2}=0$, we get
\[
\begin{aligned}
& i\partial\overline{\partial}(u\omega^{n-2})\wedge\omega \\
& = i\partial\overline{\partial}u\wedge\omega^{n-1}+ (n-2)ui\partial\overline{\partial}\omega\wedge \omega^{n-2}-(n-2)(n-3)ui\overline{\partial}\omega\wedge\partial\omega\wedge \omega^{n-3} \\
& = \frac{1}{n}(\Delta_\omega u)\omega^n+(n-2)u(i\partial\overline{\partial}\omega\wedge \omega^{n-2}-(n-3)i\overline{\partial}\omega\wedge\partial\omega\wedge \omega^{n-3}).
\end{aligned}
\]
Now, applying the operator $\partial$ to the identity $\overline{\partial}\omega\wedge\omega^{n-2}=0$, we get
\[
0=i\partial\overline{\partial}\omega\wedge\omega^{n-2}-(n-2)i\overline{\partial}\omega\wedge\partial\omega\wedge\omega^{n-3},
\]
that is
\[
i\partial\overline{\partial}\omega\wedge\omega^{n-2}=(n-2)i\overline{\partial}\omega\wedge\partial\omega\wedge\omega^{n-3},
\]
giving us
\[
i\partial\overline{\partial}(u\omega^{n-2})\wedge\omega=(\Delta_\omega u)\omega^n+ui\partial\overline{\partial}\omega\wedge\omega^{n-2}.
\]
On the other hand, by definition of $d^c$ it holds $dd^c\omega=2i\partial\overline{\partial}\omega$, and applying formula [(2.13) from \cite{AI}] in the balanced case and using the definition of the Hodge-$\ast$ operator we get
\[
2|\partial\omega|_\omega^2\omega^n=\langle dd^c\omega,\omega^2 \rangle_\omega \omega^n=dd^c\omega \wedge \ast_\omega \omega^2=n!2i\partial\overline{\partial}\omega\wedge\omega^{n-2},
\]
from which we finally obtain the \textit{linearized balanced Monge-Ampère type} operator
\[
Lu=\frac{1}{n-1}\left( \Delta_\omega u+\frac{1}{n-1}|\p\omega|_\omega^2u\right),
\]
and we can clearly notice that it is bounded (using Remark \ref{michelsohn}) and $L^2$-self adjoint.
\end{proof}

\subsection{Weighted analysis}

Our aim is now to study the invertibility of the linear operator $L$, and we wish to do this in suitable weighted functional spaces. In order to introduce said spaces we shall start by introducing a weight function useful in our situation, and for simplicity we may assume that the neighbourhood of $x$ on which the $z$ coordinates are defined contains the region $\{r(z)\leq 1\}$ (this is true up to a rescaling). Define then
\[
\rho=\rho_\varepsilon(z):=
\begin{cases} 
\varepsilon^{p+q} & \text{on} \> r(z)\leq \varepsilon^{p+q}, \\
\text{non decreasing} & \text{on} \> \varepsilon^{p+q} \leq r(z) \leq 2\varepsilon^{p+q}, \\
r(z) & \text{on} \> 2\varepsilon^{p+q} \leq r(z) \leq 1/2, \\
\text{non decreasing} & \text{on} \> 1/2 \leq r(z) \leq 1,\\
1 & \text{on} \> r(z) \geq 1,
\end{cases}
\]
Using this weight function we can introduce the weighted Hölder norm and its corresponding weighted Hölder spaces $C_{\varepsilon,b}^{k,\alpha}(M)$, where $k\geq 0$, $\alpha \in (0,1)$ is the Hölder constant, $b \in \mathbb{R}$ is the weight and $\varepsilon$ indicates the dependence on the metric $\omega$ obtained by the gluing construction done above. We define
\[
\begin{aligned}
||u||_{C_{\varepsilon,b}^{k,\alpha}(M)}:= & \sum_{i=0}^k \sup_M |\rho^{b+i}\nabla_\varepsilon^iu|_\omega \\ 
& +\underset{d_\varepsilon(x,y) < inj_\varepsilon}{\sup}\left|\min\left(\rho^{b+k+\alpha}(x),\rho^{b+k+\alpha}(y)\right)\frac{\nabla_\varepsilon^ku(x)-\nabla_\varepsilon^ku(y)}{d_\varepsilon(x,y)^\alpha}\right|_\omega,
\end{aligned}
\]
where $inj_\varepsilon$ is the injectivity radius of the metric $\omega$, and thus interpret $F$ (and $L$) as operators defined as $F: C_{\varepsilon,b}^{2,\alpha}(M) \rightarrow C_{\varepsilon,b+2}^{0,\alpha}(M)$. The definition can be extended to tensors in the same way, using the metric naturally induced on tensor bundles.\\

Following then the literature, we first wish to prove the following estimate. We shall however make an initial remark about the hypothesis we will add.
\begin{remark}
    In the following Lemma, we will assume the injectivity of the operator $L$ on the background singular manifold. This is necessary in order to have the injectivity of the operator on the "singular side", and examples in which this happens where recently found in the upcoming paper \cite{FG} by the first author and Elia Fusi. As described in \cite{FG}, we are able to produce on the same manifold metrics for which the injectivity assumption is verified and not verified, and we notice that this is connected to some volume preservation property of the deformation in the chosen ansatz. This gives this kernel a strong geometric nature and hence we expect to be possible to work modulo it (either by restricting the functional space or modifying the equations), or possibly underlying some deeper property to be satisfied by the complex manifold.
\end{remark}
\begin{lemma}\label{estimate}
With the same notations as above, for every $b \in (0,n-1)$ it exists $c>0$ (independent of $\varepsilon$) such that for sufficiently small $\varepsilon$ it holds
\[
||u||_{C_{\varepsilon,b}^{2,\alpha}} \leq c||Lu||_{C_{\varepsilon,b+2}^{0,\alpha}},
\]
for all $u \in C_{\varepsilon,b}^{2,\alpha}$, provided that the operator $L$ has no kernel on the singular manifold $\tilde{M}$.
\end{lemma}
\begin{proof}
Suppose by contradiction that the above inequality does not hold. This means that for all $k \in \mathbb{N}$ we can find $\varepsilon_k >0$ and $u_k \in C_{\varepsilon_k,b}^{2,\alpha}$ such that $\varepsilon_k \rightarrow 0$ as $k \rightarrow 0$, $||u_k||_{C_{\varepsilon_k,b}^{2,\alpha}}=1$ and 
\begin{equation}\label{hpas}
||L u_k||_{C_{\varepsilon_k,b+2}^{0,\alpha}} < \frac{1}{k}.
\end{equation}
In the first place we analyze what happens on $\tilde{M}_{reg}$, i.e. away from the exceptional part. The properties of the sequence $\{u_k\}_{k \in \mathbb{N}}$ guarantee us that we can apply Arzela-Ascoli's Theorem, and hence up to subsequences we may assume $u_k \rightarrow u_\infty$ uniformly on compact subsets of $\tilde{M}_{reg}$ in the sense of $C_b^{0,\alpha}$, with respect to $\tilde{\omega}$. Moreover, since for any compact set $K \subseteq \tilde{M}_{reg}$ there exists $n_K \in \mathbb{N}$ such that for all $k\geq n_K$ on $K$ it holds $\omega=\tilde{\omega}$, and hence $\nabla_\omega=\nabla_{\tilde{\omega}}$, we actually have $C_b^{2,\alpha}$-convergence (again uniformly on compact subsets of $\tilde{M}_{reg}$). Hence, using the hypothesis we can conclude that $u_\infty \equiv 0$.\par 
Let now $M_c:=\{r(z) \geq 1/2\} \subseteq \tilde{M}_{reg}$ be a compact set on which we know that $u_k \rightarrow 0$ uniformly in $C_b^{2,\alpha}$. To obtain a contradiction we want to prove that $\{u_k\}_{k\in \mathbb{N}}$ admits a subsequence uniformly convergent to zero in $C_b^{2,\alpha}$ also on $A:=\{r(z)<1/2\}$.\par
In order to work in this region, it is simpler to shift to the "large" coordinates $\zeta$, i.e. the coordinates on the crepant resolution $\hat{X}$ away from the exceptional part. It is then useful to recall the relations
\[
\zeta=\varepsilon^{-(p+q)}z \quad \text{and} \quad r(z)=\varepsilon^{p+q}r(\zeta),
\]
from which we can write down the explicit identification
\[
\left\{r(z)<\frac{1}{2}\right\}=A \simeq \tilde{A}=\tilde{A}_\varepsilon=\left\{r(\zeta) <\frac{1}{2}\varepsilon^{-(p+q)}\right\}\subseteq \hat{X};
\]
this last set $\tilde{A}$ is the one we will be working on.\par
The first thing to do is rewrite the weight function in terms of this coordinates on $\tilde{A}$, resulting in
\[
\rho=
\begin{cases} 
\varepsilon^{p+q} & \text{on} \> r(\zeta)\leq 1, \\
\text{non decreasing} & \text{on} \> 1 \leq r(\zeta) \leq 2, \\
\varepsilon^{p+q}r(\zeta) & \text{on} \> 2 \leq r(\zeta) \leq 1/2\varepsilon^{-(p+q)}.
\end{cases}
\]
Notice that the entire gluing region of the metric (from the previous step) is entirely contained inside the third region, i.e. $\{2 \leq r(\zeta) \leq 1/2\varepsilon^{-(p+q)}\}$.\par
We now go back to our sequence $\{u_k\}_{k\in \mathbb{N}}$. Since $||u_k||_{C_{\varepsilon_k,b}^{2,\alpha}}=1$ for all $k \in \mathbb{N}$, we have in particular that on all $\tilde{A}_k:=\tilde{A}_{\varepsilon_k}$ holds
\[
|\rho^b u_k|\leq c.
\]
Introducing then the new sequence
\[
U_k:=\varepsilon_k^{b(p+q)}u_k,
\]
the above weighted estimates for $u_k$ imply the following ones for this new sequence:
\[
\begin{cases} 
|U_k|\leq c & \text{on} \> r(\zeta)\leq 1, \\
|U_k|\leq c & \text{on} \> 1 \leq r(\zeta) \leq 2, \\
|U_k|\leq cr^{-b}(\zeta) & \text{on} \> 2 \leq r(\zeta) \leq 1/2\varepsilon_k^{-(p+q)}.
\end{cases}
\]
These estimates for $U_k$ suggest us to introduce a new weight function $\tilde{\rho}=\tilde{\rho}_k$ on $\tilde{A}_k$ given by
\[
\tilde{\rho}(\zeta)=
\begin{cases} 
1 & \text{on} \> r(\zeta)\leq 1, \\
\text{non decreasing} & \text{on} \> 1 \leq r(\zeta) \leq 2, \\
r(\zeta) & \text{on} \> 2 \leq r(\zeta) \leq 1/2\varepsilon_k^{-(p+q)},
\end{cases},
\]
with which we get that
\begin{equation}\label{wnorm}
|\tilde{\rho}^bU_k|\leq c,
\end{equation}
and analogous weighted estimates also for $\nabla U_k$ and $\nabla^2 U_k$, hence again by Ascoli-Arzelà theorem we have that $U_k \rightarrow U_\infty$ uniformly on compact sets of $\hat{X}$ (since $\tilde{A}_k \rightarrow \hat{X}$) in the sense of $\tilde{C}_b^{2,\alpha}=C_b^{2,\alpha}(\tilde{\rho})$, where this last space is the weighted Hölder space on $\hat{X}$ identified by the weight $\tilde{\rho}$ and the metric $\omega_{ALE}$.\par
On the other hand, on any compact subset of $\hat{X}$, for sufficiently large $k$ it holds
\begin{equation}\label{linresc}
\rho^{b+2}L u_k =\tilde{\rho}^{b+2}\Delta_{\omega_{ALE}}U_k,
\end{equation}
and since $\frac{1}{k}>||\tilde{L}u_k||_{C_{\varepsilon_k,b+2}^{0,\alpha}}$, taking the limit in \eqref{linresc} we obtain that $U_\infty$ is harmonic with respect to the ALE metric $\omega_{ALE}$. Moreover, taking the limit in \eqref{wnorm} ensures us that $U_\infty$ decays at infinity, from which follows that $U_\infty \equiv 0$ on the whole $\hat{X}$, and thus $U_k \overset{\tilde{C}_b^{2,\alpha}}{\rightarrow} 0$ uniformly on compact sets of $\hat{X}$.\par
If we are now able to prove that $U_k$ admits a subsequence converging uniformly to zero on the whole $\hat{X}$ in the sense $\tilde{C}_b^0$ we get our contradiction, and we are done. Indeed, if $U_k \overset{\tilde{C}_b^0}{\rightarrow} 0$ uniformly (up to subsequences) on $\hat{X}$, then scaled Schauder estimates imply that also $U_k \overset{\tilde{C}_b^{2,\alpha}}{\rightarrow} 0$ uniformly, which is the same as saying $u_k \overset{C_{\varepsilon_k,b}^{2,\alpha}}{\rightarrow} 0$ uniformly on $\{r(z)< 1/2\}$. Thus $\{u_k\}_{k \in \mathbb{N}}$ up to subsequences is uniformly convergent to zero on the whole manifold $M$, which is a contradiction with the fact that $||u_k||_{C_{\varepsilon_k,b}^{2,\alpha}}=1$ for all $k \in \mathbb{N}$.\par
Now we will prove that the said uniformly convergent subsequence exists. If by contradiction this was not the case, since we have the uniform convergence on compact sets, we would be able to find $\delta >0$ and $\{x_k\}_{k \in \mathbb{N}} \subseteq \hat{X}$, $x_k \in \tilde{A}_k$, such that $R_k:=r(\zeta(x_k)) \rightarrow +\infty$ and $R_k^b U_k(x_k) \geq \delta$ for all $k \in \mathbb{N}$, and since $R_k \rightarrow +\infty$, we can actually assume $\tilde{\rho}\equiv r$ on the points of the sequence, from which we get that for all $k \in \mathbb{N}$ holds
\begin{equation}
R_k^b|U_k(x_k)| \geq \delta.
\end{equation}
Naming then $r_k:=r(z(x_k))$, recalling the relation between the two coordinates we have $\frac{1}{2} \geq r_k=\varepsilon_k^{p+q}R_k$, thus up to subsequences we can end up into two cases:
\begin{itemize}
\item[(i)] if $r_k \rightarrow l >0$, then $x_k \rightarrow x_\infty$, and since $u_k$ is uniformly convergent on compact sets on $\tilde{M}_{reg}$, we get that $u_k(x_k)$ is bounded, giving
\[
0<\delta \leq R_k^bU_k(x_k)=(R_k\varepsilon_k^{p+q})^bu_k(x_k) = r_k^b u_k(x_k) \underset{k\rightarrow \infty}{\longrightarrow} 0,
\]
which is a contradiction;
\item[(ii)] if $r_k \rightarrow 0$, let $X^*:=\hat{X}\setminus E$ the singularity model and $X'$ a copy of $X^*$, and we consider the biholomorphisms $\sigma_k : B_k \rightarrow A\setminus \{0\}$, given by
\[
\sigma_k(z'):=r_nz',
\]
where $B_k:=\{0<r(z')<\frac{r_k^{-1}}{2}\}\subseteq X'$. Then, if we endow $B_k$ with the metric
\[
\theta_k:=r_k^{-2}\sigma_k^*\omega,
\] 
it is easy to notice that the couple $(B_k,\theta_k)$ converges to $(X',\omega_{flat})$, i.e. the standard singularity model. If we then introduce the functions 
\[
w_k:= r_k^b\sigma_k^*u_k
\]
on $B_k$, we notice that the pullback of the weight function $\rho$ gives
\[
\rho'(z')=\sigma_k^*\rho(z')=
\begin{cases} 
\varepsilon_k^{p+q} & \text{on} \> r(z')< R_k^{-1}, \\
\text{non decreasing} & \text{on} \> R_k^{-1} \leq r(z') \leq 2R_k^{-1}, \\
r_k r(z') & \text{on} \> 2R_k^{-1} \leq r(z') < \frac{r_k^{-1}}{2}, 
\end{cases}
\]
from which we get (pulling back the inequality $\rho^b|u_k| \leq 1$)
\begin{equation}\label{limw}
r^b(z')w_k(z') \leq 1 
\end{equation}
on each $z' \in X$ (assuming $k$ to be sufficiently large). Hence, this shows that for any compact $K \subseteq X'$, we can choose $k \in \mathbb{N}$ to be sufficiently large in order to have $K \subseteq B_k$ and $\rho'(z')= r_k r(z')$ on the whole $K$, and get that $w_k$ is uniformly bounded on $K$; and since this works for any compact $K\subseteq X'$, we obtain that - up to subsequences - $\{w_k\}_{k\in\mathbb{N}}$ converges uniformly on compact sets of $X'$ to a function $w_\infty$, and from \eqref{limw} we get that $w_\infty$ is decaying at infinity. Moreover, recalling that $R_k^bU_k(x_k)\geq \delta$ for all $k \in \mathbb{N}$, if we introduce the sequence $y_k:=\sigma_k^{-1}(x_k)$, it is straighforward to notice that from its definition follows that $w_k(y_k)\geq \delta$ and $||y_k||_{\theta_k}=1$ for all $k \in \mathbb{N}$, thus implying that - up to subsequences - $y_k \rightarrow y_\infty \in X'$, and hence
\begin{equation}\label{assurdo}
w_\infty(y_\infty)>0.
\end{equation}
Now, if we recall the definition of the operator $L$ and take the pullback with respect to $\sigma_k$ of $\rho^{b+2}L u_k$, it is immediate to see that on every compact $K \subseteq X'$ we get
\begin{equation}\label{pbop}
\begin{aligned}
\sigma_k^*\left(\rho^{b+2}L u_k\right) & =  \frac{n}{n-1} r^{b+2}(z')  \left(\frac{i\p\bpa w_k  \wedge \theta_k^{n-1}}{\theta_k^n}+|d\theta_k|_{\theta_k}^2w_k\right) \\  
 & =\frac{n}{n-1} r^{b+2}(z')\Delta_{\theta_k}w_k+|d\theta_k|_{\theta_k}^2w_k,
\end{aligned}
\end{equation}
from which we have, taking the limit as $k \rightarrow + \infty$, that
\[
\Delta_{\omega_{flat}}w_\infty\equiv 0 \quad \text{on} \> X',
\]
i.e., $w_\infty$ is harmonic on $X'$ with respect to the flat metric. Thus, since it decays at infinity, we obtain $w_\infty \equiv 0$, which is a contradiction as \eqref{assurdo} holds.
\end{itemize}
Thus the proof is complete.
\end{proof}

As a direct consequence we get
\begin{lemma}\label{iso}
The operator $L: C_{\varepsilon,b}^{2,\alpha}(M) \rightarrow C_{\varepsilon,b+2}^{0,\alpha}(M)$ defined above is a linear isomorphism for every $b \in (0,n-1)$, provided that the operator $L$ has no kernel on the singular manifold $\tilde{M}$.
\end{lemma}
\begin{proof}
Notice that $L$ is elliptic and shares its index with the laplacian, which is zero. Moreover, by Lemma \ref{estimate} we have that $L$ is injective, thus we automatically get that $L$ is also surjective and - from \ref{estimate} - has bounded inverse, thus $L$ is a isomorphism.
\end{proof}
With this result we can now show how to reformulate the original equation as a fixed point problem.\par
In order to do this we shall consider the expansion
\[
F(\psi)=F(0)+L(\psi)+Q(\psi),
\]
and thus rewrite the balanced Monge-Ampère type equation as
\[
F(0)+L(\psi)+Q(\psi)=0,
\]
and using now Lemma \ref{iso}, we get that our equation is therefore equivalent to
\begin{equation}\label{contr}
\psi=L^{-1}(-F(0)-Q(\psi))=:N(\psi),
\end{equation}
i.e. the search for a fixed point for the operator $N: C_{\varepsilon,b}^{2,\alpha}(M) \rightarrow C_{\varepsilon,b}^{2,\alpha}(M)$. To do this, we will have to identify the open set on which we wish to apply Banach's Lemma, and show that on said open set, the operator $N$ can be restricted and gives rise to a contraction.\par
The first thing to do is the following remark.
\begin{remark}\label{distdef}
If $C,\tau>0$, and $\varphi$ is a function on $M$ such that $||\varphi ||_{C_{\varepsilon,-2}^{2,\alpha}} \leq C\varepsilon^\tau$, thanks to Lemma \ref{balancedoperator} (which provides us the extended expression of $i\p\bpa (\varphi\omega^{n-2})$) it is straightforward to see that
\[
||i\p\bpa (\varphi\omega^{n-2})||_{C_{\varepsilon,0}^{0,\alpha}} \leq ||\varphi ||_{C_{\varepsilon,-2}^{2,\alpha}}  \leq C\varepsilon^\tau,
\]
thus we are guaranteed that, choosing $\varepsilon$ to be sufficiently small, $\omega_\varphi^{n-1}>0$, and thus its $(n-1)$ root $\omega_\varphi$ exists and is a balanced metric. Moreover, we can apply again the argument used in Remark \ref{michelsohn}, and obtain that if $||\varphi ||_{C_{\varepsilon,-2}^{2,\alpha}} \leq C\varepsilon^\tau$, then
\[
|\omega_\varphi - \omega|_\omega \leq c||\varphi ||_{C_{\varepsilon,-2}^{2,\alpha}} \leq c\varepsilon^\tau,
\]
which also implies that $\omega_\varphi \rightarrow \omega$, as $\varepsilon \rightarrow 0$.
\end{remark}
Thanks to this remark, we have a suggestion on how to choose the open set on which apply Banach's Lemma, hence we introduce
\[
U_\tau:=\{\varphi \in C_{\varepsilon,b}^{2,\alpha} \>|\> ||\varphi||_{C_{\varepsilon,b}^{2,\alpha}} < \tilde{c}\varepsilon^{(p+q)(b+2)+\tau} \} \subseteq C_{\varepsilon,b}^{2,\alpha},
\]
and we notice that for every $\varphi \in U_\tau$ it holds $||\varphi||_{C_{\varepsilon,-2}^{2,\alpha}} \leq C\varepsilon^\tau$, with $C$ independent of $\varphi$ and $\varepsilon$.\par
We will now prove that on $U_\tau$, the operator $N$ is a contraction. In particular, given $\varphi_1,\varphi_2 \in U_\tau$, we want to estimate
\[
N(\varphi_1)-N(\varphi_2)=L^{-1}((\hat{Q}(\varphi_2)-\hat{Q}(\varphi_1))).
\]
To do so, we notice that by the Mean Value Theorem we can find $t \in [0,1]$ such that
\[
Q(\varphi_1)-Q(\varphi_2)=dQ_{\nu}(\varphi_1-\varphi_2)=(L_\nu-L)(\varphi_1-\varphi_2),
\]
where $\nu=t\varphi_1+(1-t)\varphi_2 \in U_\tau$, and $L_\nu$ is the linearization of $F$ at $\nu$. With the same strategy used to compute $L$ we can easily obtain an expression for $L_\nu$, and thus get
\[
(L_\nu-L)(\varphi_1-\varphi_2)=\frac{n}{n-1}\frac{(\omega_\nu-\omega)\wedge i\p\bpa ((\varphi_1-\varphi_2)\omega^{n-2})}{\omega^n}.
\]
From here, taking the norms with respect to $\omega$, we can use the fact that $\nu \in U_\tau$ together with Remark \ref{distdef}, to obtain
\[
|(L_\nu-L)(\varphi_1-\varphi_2)| \leq c|\omega_\nu - \omega|_\omega | i\p\bpa ((\varphi_1-\varphi_2)\omega^{n-2})|_\omega \leq c \varepsilon^\tau | i\p\bpa ((\varphi_1-\varphi_2)\omega^{n-2})|_\omega,
\]
and thus, by multiplying the inequality with $\rho^{b+2}$, get
\begin{equation}\label{contr1}
||Q(\varphi_1)-Q(\varphi_2)||_{C_{b+2,\varepsilon}^{0,\alpha}} \leq c\varepsilon^\tau ||\varphi_1-\varphi_2||_{C_{\varepsilon,b}^{2,\alpha}},
\end{equation}
hence, choosing $\varepsilon$ sufficiently small ensures us that $N$ is a contraction on $U_\tau$.\par
We are left with proving that $N(U_\tau) \subseteq U_\tau$. To do this we shall assume that $pm-q(b+2)>\tau>0$ (which can easily be done), and see that for every $\varphi \in U_\tau$, thanks to estimate \eqref{contr1} and Lemma \ref{estimate}, we have
\[
\begin{aligned}
||N(\varphi)|||_{C_{\varepsilon,b}^{2,\alpha}} \leq & ||N(\varphi)-N(0)|||_{C_{\varepsilon,b}^{2,\alpha}}+||N(0)|||_{C_{\varepsilon,b}^{2,\alpha}} \\
\leq & c\varepsilon^\tau ||\varphi|||_{C_{\varepsilon,b}^{2,\alpha}}+||L^{-1}(1-e^f)|||_{C_{\varepsilon,b}^{2,\alpha}} \\
\leq & c\varepsilon^\tau ||\varphi|||_{C_{\varepsilon,b}^{2,\alpha}}+||f||_{C_{\varepsilon,b+2}^{0,\alpha}} \\
\leq & c(\varepsilon^{(p+q)(b+2)+2\tau}+\varepsilon^{p(b+2)+pm}) \\
\leq & c\varepsilon^{\min\{\tau, pm-q(b+2)-\tau\}}\varepsilon^{(p+q)(b+2)+\tau} \\
\leq & \tilde{c}\varepsilon^{(p+q)(b+2)+\tau},
\end{aligned}
\]
implying that $N(U) \subseteq U$.\par
This shows that everything is into place to apply Banach's Lemma on the open set $U$ and obtain $\hat{\omega}$ a Chern-Ricci flat balanced metric $\hat{\omega}$ on $M$, thus proving Theorem \ref{main}.\par

Remark \ref{distdef} also implies:
\begin{corollario}
The couple $(M,\hat{\omega})$ Gromov-Hausdorff converges to the singular Calabi-Yau metric on $\tilde{M}_{reg}$ and, up to rescaling, to Joyce's ALE metrics nearby the exceptional curve.
\end{corollario}
\begin{proof}
    The choice of the open set $U_\tau$, combined with Remark \ref{distdef}, guarantees that $\phi \in U_\tau$ has $C^{2,\alpha}$ norm bounded from above by a power of $\varepsilon$, and hence ensures the $C^0$ smallness of the deformation term in the ansatz (as shown in Remark \ref{distdef}). Taking the Gromov-Hausdorff limit in $\varepsilon$ reduces then to taking it for $(M,\omega)$, and from the construction of $(M,\omega)$ is straightforward to see that the statement holds.
\end{proof}
% And it is interesting to notice that \textit{en passant} we have proved:
% \begin{corollario}
% On our crepant resolutions, for the balanced classes of the pre-gluing metrics, the Gauduchon conjecture for balanced metrics (see \cite{STW}) holds, i.e. in said balanced classes we can always find a Chern-Ricci flat balanced metric.
% \end{corollario}

We conclude this part with a few remarks.
\begin{remark}
In light of Remark \ref{exceptional}, Stokes' Theorem shows that - with the deformation given by the balanced Monge-Ampère type equation - the volume of the exceptional divisors remains the same as the one of the pre-gluing metric, i.e. the (scaled) volume of the ALE metric.
\end{remark}

\begin{remark}
Thanks to what is known about Joyce's ALE metrics, if we have $k \in \mathbb{N}$ orbifold singularities and we call $E_j^i$, $i=1,...,k_j$ the exceptional divisors corresponding to the resolution of the $j$-th singularity, for $j=1,...,k$, from our construction we can conclude (in the same way as in \cite{BM}) that
\[
[\omega^{n-1}]=[\hat{\omega}^{n-1}]=[\tilde{\omega}^{n-1}]+(-1)^{n-1}\varepsilon^{(2n-2)}(\sum_{i=1}^{k_j} \sum_{j=1}^k a_j^i PD[E_j^i])^{n-1},
\]
where $PD[E_j^i]$ denotes the Poincaré dual of the class $[E_j^i]$.
\end{remark}

Thus completing the proof of Theorem \ref{main}.

\begin{remark}
It is known that for a manifold which is Calabi-Yau with holomorphic volume $\Omega$, the existence of a Chern-Ricci flat balanced metric implies that $\Omega$ is parallel with respect to the Bismut connection associated to said metric. Among the other things, this implies that the restricted holonomy of the Bismut connection of Chern-Ricci flat balanced metrics is contained in $SU(n)$.
\end{remark}

\begin{remark}
Even though this construction is done to address a non-Kähler situation, it can also be applied when $\tilde{M}$ is instead Kähler (Ricci flat). In this case we know from Joyce's theorem that $M$ admits a Kähler Calabi-Yau metric $\omega_1$, hence together with the balanced class induced by our Chern-Ricci flat balanced metric $\hat{\omega}$ we also have the one induced by $\omega_1$. This two balanced classes need not be the same, however, even if they are to coincide, there is no uniqueness result that would guarantee that the two metrics have to be the same; moreover, the deformation we used in our construction does not cover the whole balanced class, hence in this case we are not even guaranteed that the two metrics are linked by our chosen deformation.
\end{remark}

\subsection{Relation to the Hull-Strominger system}
As a conclusion of the paper, we would like to briefly relate our construction to the Hull-Strominger system and how we intend to develop our research in this direction, hence we shall first quickly recall the definition of said system (for more details we refer to the notes \cite{GF}).\\ 

The framework is given by $(X,\Omega)$ a (not necessarily Kähler) Calabi-Yau manifold, and the first equation of the system, for $\omega$ a hermitian metric, is known as the \textit{dilatino equation} and is given by 
\[
d^*\omega=d^c\log||\Omega||_\omega,
\]
which easily seen to be equivalent to the \textit{conformally balanced equation}
\[
d\left(||\Omega ||_\omega \omega^{n-1}\right)=0.
\]
Hence, this last equation tells us that we need to work with balanced manifolds, and thus we see a first relation to our scenario.\\
To complete the system we need to pair the dilatino equation with two Hermite-Einstein equations for holomorphic vector bundles, thus, in the same fashion as what we had with the dilatino equation, the presence of the Hermite-Einstein equation in the Hull-Strominger system will limit us to consider only polystable bundles. Finally, adding one last equation, known as the \textit{Bianchi identity}, we can introduce the system.

\begin{definizione}
Given a Calabi-Yau manifold $(Y,\Omega)$ and a holomorphic vector bundle $E$ on $Y$, we say that the triple $(\omega,\>h,\> \overline{\partial}_T)$ is a solution of the \textit{Hull-Strominger system} if it satisfies
\[
\begin{aligned}
\Lambda_\omega F_h & = 0, \\ 
\Lambda_\omega R & = 0, \\
d^*\omega-d^c\log ||\Omega ||_\omega & = 0, \\
dd^c\omega-\alpha(\text{tr}R\wedge R-\text{tr}F_h\wedge F_h ) & =0;
\end{aligned}
\]
where, $\alpha$ is a non-zero constant, $\omega$ is a hermitian metric on $Y$, $h$ is a hermitian metric along the fibers of $E$, $\overline{\partial}_T$ is a holomorphic structure on the tangent bundle of $Y$, and $R$ is the Chern curvature tensor of $\omega$, read as a hermitian metric on the holomorphic vector bundle $(TY,J,\overline{\partial}_T)$.
\end{definizione}

The \textit{Bianchi identity} (also known as \textit{anomaly cancellation equation}), is the hardest and least understood equation of the system and also the one we wish to address in the development of our research.\\
It is significant to notice that if we choose $(E,h)$ to be the holomorphic tangent bundle with the metric $\omega$, and take $\omega$ a Kähler Ricci-flat metric, we see that this satisfies the system, thus being a solution of the Hull-Strominger system is a condition that generalizes being Kähler Calabi-Yau, hence a very promising candidate class of \textit{special metrics}.\\
Also, thanks to the equivalence between Hermite-Einstein metrics and Hermite-Yang-Mills connections it is possible to rewrite the system from a gauge-theoretical point of view.
\begin{definizione}
Given a Calabi-Yau manifold $(Y,\Omega)$ and a hermitian vector bundle $(E,h)$ (with a fixed holomorphic structure) on $Y$, the triple $(\omega,\> A,\> \nabla)$ is a solution of the \textit{Hull-Strominger system} if it satisfies
\[
\begin{aligned}
\Lambda_\omega F_A & = 0, \quad F_A^{0,2}=0, \\
\Lambda_\omega R_\nabla & = 0, \quad R_\nabla^{0,2}=0, \\ 
d\left( ||\Omega ||_\omega \omega^{n-1}\right) & = 0, \\
dd^c\omega-\alpha(\text{tr}R\wedge R-\text{tr}F_h\wedge F_h ) & =0;
\end{aligned}
\]
where $\alpha$ is a non-vanishing constant, $\omega$ a hermitian metric on $Y$, $A$ is unitary connection on $(E,h)$ and $\nabla$ is a unitary connection on $(TM,J,g)$.
\end{definizione}
This second description is useful to notice a series of necessary conditions when $X$ is compact. Indeed, as already observed we have that $Y$ has to be necessarily balanced, and given the natural balanced class $\tau$ given by the dilatino equation, we have that $E$ and $TY$ have to be $\tau$-polystable. Morever, we have that $c_1(Y)=0$, and also
\begin{equation}\label{yauconj}
\begin{aligned}
c_1(E)\cdot \tau & =0 \\
ch_2(E) & =ch_2(Y) \in H_{BC}^{2,2}(Y,\mathbb{R}),
\end{aligned}
\end{equation}
where $ch_2$ denotes the second Chern character.\\
Although several examples of solutions have been studied over time (several can be found, for example, in \cite{GF}), there is still a very poor understanding on the existence of solutions, even on threefolds, where however it was conjectured by Yau in \cite{Y} that conditions \eqref{yauconj} are not only a necessary condition but even a sufficient one to the existence of solutions in the case of threefolds.\\ 

If we then take our construction, and we view it in the system's scenario, we can make the following final remark in which we explain our ideas on how to expand our construction in this direction.
\begin{remark}\label{applicationhs}
Given $\hat{\beta}$ a Chern-Ricci flat balanced metric on a Calabi-Yau threefold $(Y,\Psi)$, it holds
\[
||\Psi||_{\hat{\beta}} \equiv const.,
\]
showing that our metric $\hat{\omega}$ gives a solution of the conformally balanced/dilatino equation on our crepant resolutions $(M,\Omega)$. Thus our construction gives us two solutions of the dilatino equation on $(M,\Omega)$, that are $\hat{\omega}$ and $\omega':=||\hat{\Omega}||_\omega^{-2}\omega$, where this last one is the dilatino equation solution associated to the balanced metric $\omega$ obtained in the first part of the gluing construction. From here, thanks to the fact that this metrics are nearby a Kähler Ricci-flat metric, an idea could be to try and adapt strategies as in \cite{CPY1} or \cite{DS} to construct a Hermite-Einstein metric on the tangent bundle with respect to the above metrics, and eventually from there try and extend it to a whole solution of the Hull-Strominger system, using - for example - some version of the approach of \cite{AGF}.\\
Other possible paths could instead be related to the orbifold examples from \cite{BTY} and \cite{ST} we recalled in Section 2, on which it could be interesting to see if, again through a gluing process, if it is possible to construct new non-Kähler solutions to the Hull-Strominger system.
\end{remark}

\section{The conifold singularity case}

As anticipated in the introduction, it is natural to ask weather or not the construction can be adapted to the case of ordinary double points on threefolds, in order to fit our result in the conifold transition framework. Unfortunately issues show up, hence in the following we shall - after recalling the ingredients on Ordinary Double Points on threefolds - walk through our construction and see what continues to hold, see what fails and discuss ideas on how to eventually solve the issues.

\subsection{Ordinary Double Points and their small resolutions}

The type of singularity adressed in this case is the one of Ordinary Double Points on threefolds (which are the most common kind of singularities), and are described by the model
\[
X:=\{ z_1^2+z_2^2+z_3^2+z_4^2=0\} \subseteq \mathbb{C}^4,
\]
which is known as the $3$-dimensional \textit{standard conifold}, whose only singular point is the origin. Then we have:

\begin{definizione}
A singular point $p$ in a singular threefold $Y$ is called \textit{ordinary double point} (ODP) if we can find a neighborhood $p\in U\subseteq M$ and a neighborhood $0 \in V \subseteq X$ such that $U$ and $V$ are biholomorphic through a map that sends $p$ to $0$.
\end{definizione}
These singularities arise naturally on threefolds when collapsing $(-1,-1)$-curves, i.e. rational curves biholomorphic to $\mathbb{P}^1$ whose normal bundle is isomorphic to $\mathcal{O}_{\mathbb{P}^1}(-1)^{\oplus 2}$, and actually this procedure to obtain ODPs covers all the possibilities on threefolds. Indeed, the standard conifold can be constructed in several ways, one of which is the following: consider the rank $2$ bundle $\mathcal{O}_{\mathbb{P}^1}(-1)^{\oplus 2}$ on $\mathbb{P}^1$ and notice that the map
\[
([X_1:X_2],(w_1,w_2)) \mapsto (w_1X_1,w_1X_2,w_2X_1,w_2X_2)
\]
maps $\mathcal{O}_{\mathbb{P}^1}(-1)^{\oplus 2}$ onto $X$ - since $X$ through a change of coordinates is biholomorphic to the set $\{W_1 W_2 - W_3 W_4=0\}$ - sending the zero section onto the origin. Moreover this map restricted to $\mathcal{O}_{\mathbb{P}^1}(-1)^{\oplus 2}\setminus \mathbb{P}^1$ (where $\mathbb{P}^1$ is meant as the zero section) gives a biholomorphism with $X\setminus \{0\}$, proving our previous statement. This shows us that these singularities always admit small resolutions (with $\mathbb{P}^1$ as the exceptional curve) biholomorphic to $\hat{X}:=\mathcal{O}_{\mathbb{P}^1}(-1)^{\oplus 2}$, and it can be shown that a singular threefold with $n$ ordinary double points admits exactly $2^n$ small resolutions of this type (every singularity can be resolved with a curve in two distinct bimeromorphic ways).\\

Regarding instead the metric aspect of this singularities, the standard conifold $X$ is naturally endowed with a conical structure. Indeed, we can introduce the function on $\mathbb{C}^4$
\[
r(z):=||z||^{\frac{2}{3}},
\]
which restricted to $X$ yields the conical distance to the singularity, and can be used to define the metric
\[
\omega_{co,0}:=\frac{3}{2}i\partial\overline{\partial}r^2,
\]
on the smooth part of $X$, which is clearly Kähler. Moreover, it can be seen that $\omega_{co,0}$ is actually also Ricci flat, as well as a cone metric over the link $L:=\{r=1\}\subseteq X$ which can be written as
\[
g_{co,0}=\frac{3}{2}(dr^2+r^2g_L),
\]
with $g_L$ a Sasaki-Einstein metric on the link $L$.\\
This metric structure of the standard conifold, with some further work, yields also a Kähler Calabi-Yau structure on the small resolution. In fact, Candelas and de la Ossa (see \cite{CO}) constructed a family of metrics (depending on a parameter $a>0$) via the ansatz 
\begin{equation}\label{anscdo}
\omega_{co,a}:=i\partial\overline{\partial}f_a(r^3)+4a^2\pi_{\mathbb{P}^1}^*\omega_{FS}
\end{equation}
on the conifold $X$, where $\omega_{FS}$ is the Fubini-Study metric on $\mathbb{P}^1$. Imposing Kähler-Ricci flatness in this ansatz resulted in $f_a$ being a smooth function satisfying the ODE
\[
(xf_a'(x))^3+6a^2(xf_a'(x))^2=x^2, \qquad f_a(x) \geq 0,
\]
on $[0,+\infty)$ (which immediately gives $f_a(x)=a^2f_1(x/a^3)$), which allowed the authors to show that the metric constructed smoothly extends to the whole small resolution $\hat{X}$. Here the function $r$ is simply the conical distance from the singularity re-read on the resolution, hence portraying the conical distance from the exceptional curve. Moreover, this family of metrics is such that as $a\rightarrow 0$ the metrics $\omega_{co,a}$ converges, away from the exceptional curve, to the standard cone metric $\omega_{co,0}$, and it is also asymptotic (at infinity) to the cone metric $\omega_{co,0}$, and these facts can be seen explicitly with the following expansion from \cite{CPY1}.

\begin{lemma}\label{expansion}
For $x\gg 1$, the function $f_1(x)$ has a convergent expansion
\[
f_1(x)=\frac{3}{2}x^{\frac{2}{3}}-2\log(x)+\sum_{n=0}^{+\infty}c_n x^{-\frac{2n}{3}}.
\]
\end{lemma}

We shall now move on towards the gluing attempt.

\subsection{Gluing attempt and possible solutions}

First of all, we lay out the details of the setting and take $\tilde{M}$ a \textit{smoothable} Kähler Calabi-Yau singular threefold obtained from the contraction of a finite family of disjoint $(-1,-1)$-curves in a compact complex threefold (thus the singular set of $\tilde{M}$ is made of a finite number of Ordinary Double Points) - hence with the regular part $M_{reg}$ of $\tilde{M}$ equipped with $\tilde{\omega}$ a Kähler Calabi-Yau metric - and $M$ a compact small resolution of $\tilde{M}$.\\

\begin{remark}
The reason why we have much stronger assumptions with respect to the orbifold case, is because for this type of singularities we are not aware of a version of Lemma \ref{flatcutoff}, thus we need a condition to be able to smoothly cut-off the singular metric at the standard model around the singularity (given in this case by the standard cone metric $\omega_{co,0}$), and such condition is given exactly by the smoothability assumption, which allows us to apply the following result from Hein and Sun (Theorem 1.4 and Lemma A.1 from \cite{HS}), which can be simplified for our purpose with the following statement (here written only for threefolds):
\begin{teorema}[Hein-Sun]\label{hs}
Let $\tilde{M}$ be a smoothable singular threefold whose singular set is a finite family of ODPs endowed with a Kähler Calabi-Yau metric $\tilde{\omega}$ on its smooth part $M_{reg}$. Then for every singular point $p \in \tilde{M}\setminus M_{reg}$ there exist a constant $\lambda_0 >0$, neighborhoods $p \in U_p \subseteq \tilde{M}$ and $0 \in V_p \subseteq X$, and a biholomorphism $P: V_p \setminus \{0\} \rightarrow U_p \setminus \{p \}$ such that
\[
P^*\tilde{\omega}-\omega_{co,0}=i\partial\overline{\partial} \varphi, \qquad \text{for some} \> \varphi \in C_{2+\lambda_0}^{\infty},
\]
where $r$ is the conical distance from the singularities and $C_{2+\lambda_0}^{\infty}$ is the space of smooth functions with decay rate at zero of $2+\lambda_0$ (i.e. an $f \in C_{2+\lambda_0}^{\infty}$ is a smooth function such that nearby zero it holds $|\nabla^k f| \leq cr^{2+\lambda_0-k}$ for all $k \geq 0$).
\end{teorema}
Anyway, what follows actually works if we replace the assumption above with: $\tilde{\omega}$ a singular Chern-Ricci flat balanced metric such that in a neighborhood of each singularity is asymptotic to the standard cone metric $\omega_{co,0}$.
\end{remark}

Now, since our work aims to face the case of compact non-Kähler small resolutions of said $\tilde{M}$, before describing the gluing attempt it is significant to show that this kind of resolutions are actually a very common situation.

\begin{remark}
Thanks to a result from Cheltsov (see \cite{Ch}) we know that a hypersurface $\tilde{M}$ in $\mathbb{P}^4$ of degree $d$ with only isolated ODPs is factorial when $\tilde{M}$ has at most $(d-1)^2-1$ singularities, thus is in particular $\mathbb{Q}$-factorial. On the other hand $\tilde{M}$, being a hypersurface is clearly smoothable (in more general situations, one can apply the work from Namikawa and Steenbrink - see \cite{NS} - to obtain again the smoothability), and hence, thanks to the results from Friedman (see \cite{F}) we have that any small resolution $M$ of $\tilde{M}$ with exceptional curves $C_1,...,C_k$, $C_i \simeq \mathbb{P}^1$, satisfies necessarily a condition
\[
\sum_{i=1}^k\lambda_i [C_i]=0 \quad \text{in} \> \> H_2(M,\mathbb{R}), \quad \text{where each}\> \> \lambda_i\neq 0,
\]
which immediately implies that if $\tilde{M}$ has only one ODP, then $M$ can't be Kähler because it contains a homologically trivial curve (notice that the generic singular quintic threefold is expected to have exactly one node, and is smoothable since it is a hypersurface in the projective space, hence satisfies this scenario).\\ 
Moreover, Werner proved in \cite{W} that $M$ is projective if and only if all $C_i$s are homologically non-trivial, and since $M$ is Moishezon, projectivity is equivalent to Kählerness. Thus the class of examples above lies in a larger one, since every small resolution with at least a homologically trivial exceptional curve is non-Kähler.
\end{remark}

Before discussing the construction, if we momentarily drop the curvature condition, it is straightforward from literature to conclude the existence of balanced metrics on the small resolution. Indeed:

\begin{remark}
Thanks to the results from Hironaka and Alessandrini-Bassanelli (\cite{Hi} and \cite{AB2}), we already know that such small resolutions admit balanced metrics, since blowing up the singularites produces a smooth Kähler threefold which is birational to the small resolution. This fact also shows that for the non-Kähler small resolutions we are considering, the Fino-Vezzoni conjecture (see \cite{FV}, Problem 3) holds true, since $M$ is Moishezon, and thus we can apply Theorems B and C from \cite{CRS} to obtain that $M$ does not admit SKT metrics.
\end{remark}

We will now present the gluing attempt. Since the proofs are essentially the same as the ones performed in Sections 2 and 3, we will just avoid them and only state the results. Again for simplicity we will just work with one singularity.\\
The first thing to do is to produce a pre-gluing metric, and in the same fashion as we have done in Step 2, we do this in three steps.
\begin{itemize}
\item[(1)] First, we glue the background singular metric $\tilde{\omega}$ to the standard cone metric around the singularity. To do so, we take a cut-off function $\chi_\varepsilon$ as in Step 1 above, and use Theorem \ref{hs}. Indeed, if we take $p>0$ and $\varepsilon>0$ sufficiently small, so that on the region $\{0<r\leq 2\varepsilon^p\}\subseteq X$ exists a constant $\lambda_0>0$ and is defined a function $\varphi \in C_{2+\lambda_0}^{\infty}$ such that
\[
\tilde{\omega}=\omega_{co,0}+i\partial\overline{\partial}\varphi,
\]
we can define the smooth real $(1,1)$-form
\[
\tilde{\omega}_\varepsilon :=\omega_{co,0}+i\partial\overline{\partial}(\chi_\varepsilon(r)\varphi),
\]
which for $\varepsilon$ sufficiently small defines a Kähler metric on $M_{reg}$ which is exactly conical around the singularity.
\item[(2)] Now, we work on the small resolution of the conifold $\tilde{X}$ and glue the Candelas-de la Ossa metric $\omega_{co,a}$ to the standard cone metric, away from the exceptional curve, and since it's not possible to do it preserving the Kähler condition, we will do it maintaining the balanced one. This can be done thanks to the fact that the Candelas-de la Ossa metric is not exact at infinity, but its square is so, since it holds
\[
\omega_{co,a}^2=\left( i\partial\overline{\partial}\left(\frac{3}{2}r^2+a^2\psi_a(r)\right)\right)^2+2a^2i\partial\overline{\partial}\left(f_a(r^3)\wedge\pi^*\omega_{FS}\right).
\]
Thus if we introduce a cut-off function $\chi_R$ as in Step 2 above, we can define the family of closed $(2,2)$-forms
\[
\omega_{a,R}^2=\left( i\partial\overline{\partial}\left(\frac{3}{2}r^2+a^2\chi_R(r)\psi_a(r)\right)\right)^2+2a^2i\partial\overline{\partial}\left(\chi_R(r)f_a(r^3)\right)\wedge\pi^*\omega_{FS},
\]
which correspond to balanced metrics for sufficiently large $R>0$.
\item[(3)] As in Step 3 above, we suitably rescale the metrics $\omega_{a,R}$ on the bubble with a geometric parameter $\lambda$ and match the two pieces on their exactly conical regions, and hence define
\[
\omega=\omega_{\varepsilon,R}:=
\begin{cases} \lambda\omega_{a,R} \qquad & \text{on} \> r(\zeta)\leq R, \\
\omega_{co,0} \qquad & \text{on}\> \varepsilon^p \leq r(z) \leq 2\varepsilon^p, \\ 
\omega_\varepsilon \qquad & \text{on} \> r(z) \geq 2\varepsilon^p.
\end{cases}
\]
\end{itemize}
At this stage, as done above, we can just unify the parameters $\varepsilon$ and $R$, and choose $R:=\varepsilon^{-q}$, with $q>0$, and using Remark \ref{michelsohn} we can see that on the gluing region $\{\frac{1}{2}\varepsilon^p<r\leq 2\varepsilon^p\}$ holds
\[
\omega=\omega_{co,0}+O(r^{\lambda_0})+O(r^{2q/p}\log r).
\]
Moreover, we can also here match the holomorphic volumes of the singular threefold and of the small resolution to obtain an (almost) explicit holomorphic volume $\Omega$ for $M$, which can be used again to define the global Chern-Ricci potential
\[
f=f_{p,q,\varepsilon}:=\log\left(\frac{i\hat{\Omega}\wedge\overline{\hat{\Omega}}}{\omega^3}\right),
\]
and obtain that globally on $M$ holds
\[
|f|= O(r^{\lambda_0})+O(r^{2q/p}\log r),
\]
i.e. a small Chern-Ricci potential. 
\begin{remark}
As in Remark \ref{applicationhs}, the existence of this metric gives us immediately a solution to the dilatino equation, that is the metric $\omega':=||\hat{\Omega}||_\omega^{-2}\omega$, which is still quite explicit, thus again a potentially interesting starting point for the construction of a solution to the Hull-Strominger system.
\end{remark}

Let us now analyze then the cohomology class naturally associated to the metric $\omega$ just obtained, i.e. the $(2,2)$-class
\[
[\omega^2] \in H_{dR}^{2,2}(M).
\]
If we introduce two cut off functions $\theta_1,\theta_2 : [0,+\infty) \rightarrow [0,1]$ defined as follows:
\[
\theta_1(x):=
\begin{cases}
1 & \text{if} \>\> x\leq\frac{1}{8}\varepsilon^{-q} \\
\text{non increasing} & \text{if} \>\> \frac{1}{8}\varepsilon^{-q} \leq x \leq \frac{1}{4}\varepsilon^{-q} \\
0 & \text{if} \>\> x \geq \frac{1}{4}\varepsilon^{-q}
\end{cases}
\]
and
\[
\theta_2(x):=
\begin{cases}
0 & \text{if} \>\> x\leq 8\varepsilon^{-q} \\
\text{non decreasing} & \text{if} \>\> 8\varepsilon^{-q} \leq x \leq16\varepsilon^{-q} \\
1 & \text{if} \>\> x \geq 16\varepsilon^{-q};
\end{cases}
\]
and since for sufficiently small $\varepsilon$ we have that $\omega$ is exact on $K:=\{\frac{1}{8}\varepsilon^{-q} \leq r(\zeta) \leq 16\varepsilon^{-q}\}$, it exists a $3$-form $\beta$ such that
\[
\omega^2=i\p\bpa\beta \quad \text{on} \> K.
\]
Introduce then the form
\[
\Omega_c:=
\begin{cases}
i\p\bpa((\theta_1(r(\zeta))+\theta_2(r(\zeta)))\beta) \quad \text{on} \> K;\\
\omega^2 \quad \text{elsewhere},
\end{cases}
\]
and notice that
\[
\beta-(\theta_1(r)+\theta_2(r))\beta
\]
can be extended as zero to the whole $M$, thanks to the definition of the cut-offs, and thus get that 
\[
[\omega^2]=[\Omega_c],
\]
i.e. the class $[\omega^2]$ can be represented by $\Omega_c$. In addition, the two cut-offs introduced also allow us to decompose $\Omega_c=\Omega_c'+\Omega_c''$, such that on $K$ hold
\[
\Omega_c'=\begin{cases} i\p\bpa(\theta_1(r)\beta) & \text{on} \> K,\\ \omega^2 & \text{on} \> r(\zeta) \leq \frac{1}{8}\varepsilon^{-q}, \\ 0 & \text{elsewhere}, \end{cases} \quad \text{and} \quad \Omega_c''=\begin{cases} i\p\bpa(\theta_2(r)\beta) & \text{on} \> K,\\ \omega^2 & \text{on} \> r(\zeta) \geq 16\varepsilon^{-q}, \\  0 & \text{elsewhere}. \end{cases}
\]
In particular, the choice of the cut-off functions allows us to interpret the forms $\Omega_c'$ and $\Omega_c''$ as compactly supported forms on (respectively) $\hat{X}$ and $M_{reg}$ (via the identifications considered throughout the paper), and from their definition it is straightforward to see that
\[
[\Omega_c']=\varepsilon^{4(p+q)}[\omega_{co,a}^2] \in H_c^{2,2}(\hat{X})
\]
and
\[
[\Omega_c'']=[\tilde{\omega}^2] \in H_c^{2,2}(M_{reg}),
\]
where $H_c$ denotes the compactly supported cohomology group.\par
Now, we can actually say more about $[\Omega_c']$:
\begin{remark}
    Using \cite[Theorem 5.1]{Go} and \cite[Theorem 3.1]{CH}, we can see that $$\omega_{co,a}=\pi^*\omega_{FS}+i\p\bpa h$$ globally on $\hat{X}$, hence with $h$ a smooth function globally defined on $\hat{X}$. This implies immediately that $\omega_{co,a}^2$ is $i\p\bpa$-exact on the whole $\hat{X}$.
\end{remark}
Thus, $[\Omega_c']=0$, and hence
\[
[\omega^2]=[p^*\tilde{\omega}^2] \in H^4(M),
\]
where $p:M \rightarrow \tilde{M}$ is the resolution map.
Finally, we also notice that 
\[
\int_{\mathbb{P}^1}\omega=\varepsilon^{2(p+q)}\int_{\mathbb{P}^1}\omega_{co,a} \underset{\varepsilon \rightarrow 0}{\longrightarrow} 0,
\]
hence the balanced class $[\omega^2]$, as $\varepsilon \rightarrow 0$, converges to a \textit{nef class}, i.e. to the boundary of the balanced cone. This completes the proof of Proposition \ref{balancedodp}.\\\\
From what was proven above, the pre-gluing metric $\omega$ appears as suitable for a deformation argument, but unfortunately it is exactly here where the issue lies, and descends from the asymptotic behaviour of the Candelas-de la Ossa metrics.\\
Indeed, we can again consider the balanced Monge-Ampère type equation \eqref{balma}, obtained with our ansatz for the Fu-Wang-Wu balanced deformation, and obtain the corresponding operator $F$ and its linearization at zero $L$ (we use the same names of the operators used above since their expressions are unchanged). At this point, considering analogous weighted Hölder spaces as the ones used in Section 3, and a variation of $F$ (following an argument of \cite{Sz}) given by $\tilde{F}(\psi):=\frac{\omega_\psi^n}{\omega^n}-e^{f-ev_x{\psi}}$ we obtain, with essentially the same proof, the invertibility of the corresponding linearization $\tilde{L}$ and an estimate for its inverse (as in Lemma \ref{estimate}), i.e.
\begin{lemma}
For every $b \in (0,2)$ it exists $c>0$ (independent of $\varepsilon$) such that for sufficiently small $\varepsilon$ the operator $\tilde{L}$ is invertible and it holds
\[
||u||_{C_{\varepsilon,b}^{2,\alpha}} \leq c||\tilde{L}u||_{C_{\varepsilon,b+2}^{0,\alpha}},
\]
for all $u \in C_{\varepsilon,b}^{2,\alpha}$.
\end{lemma}
From here, we see that we can again turn the equation $\tilde{F}(\psi)=0$ (which still produces Chern-Ricci flat balanced metrics) into a fixed point problem. In order to do this we shall introduce the operators $\hat{F}, E, G: C_{\varepsilon,b}^{2,\alpha}(M) \rightarrow C_{\varepsilon,b+2}^{0,\alpha}(M)$ defined as
\[
\hat{F}(\psi):=\frac{\omega_\psi^3}{\omega^3}, \quad E(\psi):=e^{f+ev_x(\psi)} \quad \text{and} \quad G(\psi)=e^fev_x(\psi),
\]
from which we can write
\[
\tilde{F}=\hat{F}-E.
\]
Now, we can consider the expansion
\[
\hat{F}(\psi)=\hat{F}(0)+L(\psi)+\hat{Q}(\psi),
\]
and thus rewrite $\tilde{F}(0)=0$ as
\[
\hat{F}(0)+L(\psi)+\hat{Q}(\psi)-E(\psi)=0.
\]
Here, we notice that $\tilde{L}=L-G$, thus we can rewrite $\tilde{F}(0)=0$ once more and get
\[
\hat{F}(0)+\tilde{L}(\psi)+\hat{Q}(\psi)+G(\psi)-E(\psi)=0, 
\]
and using the above Lemma, we get that the balanced Monge-Ampère type equation is therefore equivalent to
\begin{equation}
\psi=\tilde{L}^{-1}(E(\psi)-G(\psi)-\hat{F}(0)-\hat{Q}(\psi))=:N(\psi),
\end{equation}
i.e. the search for a fixed point for the operator $N: C_{\varepsilon,b}^{2,\alpha}(M) \rightarrow C_{\varepsilon,b}^{2,\alpha}(M)$.\\
At this stage, analogously as above it is easy to check that on a suitable open set $U_\tau$, with $\tau>0$, given by
\[
U_\tau:=\{\varphi \in C_{\varepsilon,b}^{2,\alpha} \>|\> ||\varphi||_{C_{\varepsilon,b}^{2,\alpha}} < \tilde{c}\varepsilon^{(p+q)(b+2)+\tau} \} \subseteq C_{\varepsilon,b}^{2,\alpha},
\]
it holds that $N$ is a contraction operator. Unfortunately, it is impossible to consistently choose $p$, $q$ and $\tau$ to repeat the above proof and obtain that
\[
N(U_\tau) \subseteq U_\tau,
\]
and this is caused by the asymptotic quadratic decay to the cone of the Candelas-de la Ossa metrics (unusual for Calabi-Yau metrics). Actually, what happens is that this quadratic decay is exactly the threshold for this argument to work, since if said decay was (arbitrarily) more than quadratic, the argument would have worked without issues.

Analyzing further the Candelas-de la Ossa metrics, one can see that if we just consider the cut-off metrics $\omega_{a,R}$ on the small resolution $\hat{X}$, these are exactly conical at infinity, thus a deformation argument as the one performed above could lead to Chern-Ricci flat balanced metrics with faster decay to the cone, but unfortunately the metric $\omega_{a,R}$ cannot be used to do this as the "initial error" given by the Chern-Ricci potential of said metric turns out to be blowing up with respect to the weighted Hölder norm, suggesting that there might not be any Chern-Ricci flat balanced metrics in a neighborhood of the Candelas-de la Ossa metrics.\\
Hence, a possibility that we wish to explore in order to solve this issue, is to understand if it is possible to obtain Chern-Ricci flat balanced metrics on $\hat{X}$ which have fast decay but are not necessarily near to the Candelas-de la Ossa metrics, and the approach we think might be interesting to use is to try and obtain a balanced version of Conlon-Hein's result (see \cite{CH}) starting from the metric $\omega_{a,R}$, which would immediately produce the missing ingredient to complete the above failed gluing construction. Obviously such a problem comes with several challenges on the analytic side, as the balanced setting and the definition of the balanced Monge-Ampère type equation do not allow many of the tools typically to obtain Yau's estimates such as the Moser iteration technique, and the non-compact (even though weighted) setting makes it also hard to apply other inequalities that are typically used in non-Kähler settings such as the Cherrier inequality (see \cite{TW1}). Another possible interesting path to take could be to try and understand if the balanced class induced by the metric $\omega$ could be a polystable class for the holomorphic tangent bundle. This, thanks to the Hitchin-Kobayashi correspondence, would lead us to the existence of Hermite-Einstein metrics on said bundle, and thus add a block in the construction of a solution to the Hull-Strominger system.


\begin{thebibliography}{100}

\bibitem[AB1]{AB1} L. Alessandrini, G. Bassanelli, {\it Modifications of compact balanced manifolds}, C. R. Acad. Sci. Paris Sér. I Math. {\bf 320}, 1517-1522 (1995).

\bibitem[AB2]{AB2} L. Alessandrini, G. Bassanelli, {\it A class of balanced manifolds}, Proc. Japan Acad., {\bf 80}, Ser. A (2004).

\bibitem[AGF]{AGF} B. Andreas, M. Garcia-Fernandez, {\it Solutions of the Strominger System via Stable Bundles on Calabi-Yau Threefolds}, Communications in Mathematical Physics \textbf{315}, 153–168 (2012)

\bibitem[AI]{AI} B. Alexandrov, S. Ivanov, {\it  Vanishing theorems on Hermitian manifolds}, Differential Geom. Appl.{\bf 14}, 251--265  (2001).

\bibitem[AP]{AP} C. Arezzo, F. Pacard, \emph{Blowing up and desingularizing constant scalar curvature Kähler manifolds}, Acta mathematica 196.2 (2006): 179-228.

\bibitem[BM]{BM} O. Biquard, V. Minerbe. \emph{A Kummer construction for gravitational instantons}, Communications in mathematical physics 308.3 (2011) 773-794.

\bibitem[BTY]{BTY} M. Becker, L.-S. Tseng, S.-T. Yau, \emph{New Heterotic non-Kähler geometries}, Adv. Theor. Math. Phys. {\bf 13} (2009) 1815–1845.

\bibitem[C]{C} E. Calabi, \emph{Extremal Kähler metrics}, Seminar on Differential Geom., Ann. of Math. Stud. {\bf 16} (1982) 259-290.

\bibitem[Ch]{Ch} I. Cheltsov, \emph{Factorial threefold hypersurfaces}, J. Algebraic Geometry {\bf 19} (2010) 781-791.

\bibitem[CH]{CH} R. J. Conlon, H.-J. Hein, \emph{Asymptotically conical Calabi-Yau manifolds, I} Duke Math. J. 162, no. 15 (2013), 2855-2902.

\bibitem[CO]{CO} P. Candelas, X. de la Ossa, \emph{Comments on conifolds}, Nuclear Phys. B {\bf 342} (1990), no. 1, 246-268.

\bibitem[CPY1]{CPY1} T. C. Collins, S. Picard, S.-T. Yau, \emph{Stability of the tangent bundle through conifold transitions},  arXiv:2102.11170v2, to appear in Comm. Pure and Appl. Math.

\bibitem[CPY2]{CPY2} T. C. Collins, S. Picard, S.-T. Yau, \emph{The Strominger system in the square of a Kähler class}, 	arXiv:2211.03784

\bibitem[CRS]{CRS} I. Chiose, R. Răsdeaconu, I. Şuvaina,  \emph{Balanced manifolds and SKT metrics}, Annali di Matematica Pura ed Applicata {\bf 201}, 2505–2517 (2022).

\bibitem[DS]{DS} R. Dervan, L. M. Sektnan, \emph{Hermitian Yang-Mills connections on blowups}, The Journal of Geometric Analysis {\bf 31} (1), pp. 516-542, 2021.

\bibitem[F]{F} R. Friedman, \emph{Simultaneous resolution of threefold double points}, Math. Ann. {\bf 274} (1986) 671-689.

\bibitem[FG]{FG} E. Fusi, F. Giusti, \emph{Blowing up Chern-Ricci flat balanced metrics}, in preparation.

\bibitem[FeY]{FeY} T. Fei, S.-T. Yau, {\it Invariant solutions to the Strominger system on complex Lie groups and their quotients}, Comm. Math. Phys. {\bf 338}, 1183--1195 (2015) 

\bibitem[FLY]{FLY} J.-X. Fu, J. Li and S.-T. Yau, \emph{Constructing balanced metrics on some families of non- Kähler Calabi-Yau threefolds}, J. Diff. Geom. {\bf 90} (2012) 81–129.

\bibitem[FuY]{FuY} J. Fu, S.-T. Yau, \emph{The theory of superstring with flux on non-Kähler manifolds and the complex Monge-Ampère equation}, J. Differential Geom. {\bf 78} (2008), no. 3, 369-428.

\bibitem[FV]{FV} A. Fino, L. Vezzoni, \emph{Special Hermitian metrics on compact solvmanifolds}, J. Geom. Phys. {\bf 91} (2015), 40-53.

\bibitem[FWW]{FWW} J. Fu, Z. Wang, D. Wu, \emph{Form-type Calabi-Yau equations}, Math. Res. Lett. {\bf 17} (2010), no. 05, 887-903.

\bibitem[G]{G} P. Gauduchon, {\it Fibr\'e hermitiens \`a endomorphisme de Ricci non n\'egativ}, Bulletin de la S.M.F. {\bf 105}, 113--140 (1977).

\bibitem[GF]{GF} M. Garcia-Fernandez, {\it Lectures on the  Strominger system}, Trav. Math. {\bf 24}, 7--61 (2016).

\bibitem[GFGM]{GFGM} M. Garcia-Fernandez, R. Gonzalez-Molina, \emph{Harmonic metrics for the Hull-Strominger
system and stability}, arXiv:2301.08236 (2023). 

\bibitem[GP]{GP} E. Goldstein and S. Prokushkin, \emph{Geometric model for complex non-Kähler manifolds with $SU(3)$ structure}, Comm. Math. Phys. {\bf 251} (2004), no. 1, 65-78.

\bibitem[Go]{Go} R. Goto, \emph{Calabi-Yau structures and Einstein-Sasakian structures on crepant resolutions of isolated singularities}, Journal of the Mathematical Society of Japan {\bf 64}, no. 3 (2012), 1005-1052.

\bibitem[Hi]{Hi} H. Hironaka, \emph{Flattening theorems in complex analytic geometry}, Amer. J. Math. {\bf 97} (1975) 503-547.

\bibitem[HS]{HS} H.-J. Hein, S. Sun, \emph{Calabi-Yau manifolds with isolated conical singularities}, Publications mathématiques de l'IHES 126.1 (2017), 73-130.

\bibitem[Hu]{Hu} C. Hull, \emph{Superstring compactifications with torsion and space-time supersimmetry}, In Turin 1985 Proceedings "Superunification and Extra Dimensions" (1986) 347-375.

\bibitem[J]{J} D. D. Joyce, \emph{Compact manifolds with special holonomy}, Oxford Mathematical Monographs, Oxford University Press, Oxford, 2000.

\bibitem[KM]{KM} J. Kollar, S. Mori, \emph{Birational geometry of algebraic varieties}, Cambridge tracts in mathematics 134 (1998).

\bibitem[LY1]{LY1} J. Li, S.-T. Yau, \emph{The existence of supersymmetric string theory with torsion}, J. Diff. Geom. {\bf 70} (2005) 143-181.

\bibitem[LY2]{LY2} J. Li, S.-T. Yau, \emph{Hermitian-Yang-Mills connections on non-Kähler manifiolds}, Mathematical aspects of string theory (San Diego, Calif., 1986) 560-573, Adv. Ser. Math. Phys., 1, Worlds Sci. Publishing.

\bibitem[LY3]{LY3} J. Li, S.-T. Yau, \emph{The existence of supersymmetric string theory with torsion}, J. Differential Geom. {\bf 70} (2005), no. 1, 143–181.

\bibitem[M]{M} M.L. Michelsohn, {\it On the existence of special metrics in complex geometry}, Acta Math. {\bf 149}, 261--295 (1982).

\bibitem[NS]{NS} Y. Namikawa, J. H. M. Steenbrink. \emph{Global smoothing of Calabi-Yau threefolds}, Inventiones Mathematicae 122 (1995) 403-419.

\bibitem[P]{P} D.H. Phong, \emph{Moduli in geometry and physics, in Geometry and Physics in Cracow}, Acta Phys. Polon. B Proc. Suppl. {\bf 4} (2011), no.3, 351–378.

\bibitem[PPZ]{PPZ} D. H. Phong, S. Picard, and X. Zhang. \emph{Geometric flows and Strominger systems}, Mathematische Zeitschrift 288.1 (2018): 101-113.

\bibitem[R]{R} M. Reid, {\it The moduli space of 3-folds with $K=0$ may nevertheless be irreducible}, Math. Ann. {\bf 278} (1987), no. 1-4, 329-334.

\bibitem[S]{S} A. Strominger, \emph{Superstrings with torsion}, Nucl. Phys. B {\bf 274} (2) (1986) 253-284.

\bibitem[ST]{ST} T. Sferruzza, A. Tomassini, \emph{Dolbeault and Bott-Chern formalities: deformations and $\partial\overline{\partial}$-Lemma},  J. Geom. Phys. {\bf 175} (2022), Paper No. 104470, 19 pp.

\bibitem[STW]{STW} G. Székelyhidi, V. Tosatti, B. Weinkove, \emph{Gauduchon metrics with prescribed volume form}, Acta Math., {\bf 219} (2017), 181-211.

\bibitem[Sz]{Sz} G. Székelyhidi, \emph{An introduction to extremal Kähler metrics}, Graduate Studies in Mathematics. Vol. 152. American Mathematical Soc. (2014).

\bibitem[TW]{TW} V. Tosatti, B. Weinkove, \emph{The complex Monge-Ampère equation on compact Hermitian manifolds}, J. Amer. Math. Soc. {\bf 23} (2010), no.4, 1187–1195.

\bibitem[TW1]{TW1} V. Tosatti, B. Weinkove, \emph{The Monge-Ampère equation for (n-1)-plurisubharmonic functions on a compact Kähler manifold}, J. Amer. Math. Soc. 30 (2017), no.2, 311-346.

\bibitem[TY]{TY} L.-S. Tseng, S.-T. Yau,  \emph{Non-Kähler Calabi-Yau manifolds}, in String-Math 2011, 241– 254, Proc. Sympos. Pure Math., 85, Amer. Math. Soc., Providence, RI, 2012.

\bibitem[W]{W} J. Werner \emph{Kleine Auflösungen spezieller dreidimensionaler Varietäten}, Bonn, Univ., Diss., 1987 (Nicht f.d. Austausch). 

\bibitem[Y]{Y} S.-T. Yau, \emph{Metrics on complex manifolds}, Science in China Series A Mathematics {\bf 53} (2010) 565-572.

\end{thebibliography}
\end{document}